\documentclass[12pt]{article}
\usepackage{amsmath,amssymb,amsthm}
\usepackage{color}
\usepackage{comment}
\numberwithin{equation}{section}
\newtheorem{thm}{Theorem}[section]
\newtheorem{prop}[thm]{Proposition}
\newtheorem{cor}[thm]{Corollary}
\newtheorem{exam}[thm]{Example}
\newtheorem{rem}[thm]{Remark}
\newtheorem{lem}[thm]{Lemma}
\newtheorem{defn}[thm]{Definition}
\newtheorem{assum}[thm]{Assumption}

\def\supp{{\rm supp}}

\begin{document}
\title{Limiting distributions for the maximal displacement of branching Brownian motions}
\author{Yasuhito Nishimori\thanks{Department of General Education,
National Institute of Technology, Anan College,
Anan, Tokushima, 774-0017, Japan; 
\texttt{nishimori@anan-nct.ac.jp}}
\ and 
Yuichi Shiozawa\thanks{Department of Mathematics,
Graduate School of Science,
Osaka University, Toyonaka, Osaka, 560-0043,
Japan; \texttt{shiozawa@math.sci.osaka-u.ac.jp}}
\thanks{Supported in part by JSPS KAKENHI No.\ JP17K05299.}}
\maketitle

\begin{abstract}
We determine the long time behavior 
and the exact order of the tail probability for the maximal displacement 
of a branching Brownian motion in Euclidean space 
in terms of the principal eigenvalue of the associated Schr\"odinger type operator.  
To establish our results, 
we show a sharp and locally uniform growth order of the Feynman-Kac semigroup.   
\end{abstract}
\section{Introduction}

We studied in \cite{S18, S19} asymptotic properties related to the maximal displacement  
for a branching Brownian motion on ${\mathbb R}^d$ 
with spatially inhomogeneous branching structure. 
In particular, we determined the linear growth rate and   
the exponential decay rate of the tail probability 
for the maximal displacement 
in terms of the principal eigenvalue of the associated Schr\"odinger type operator. 
In this paper, we investigate the second growth order and  
the exact order of the tail probability including the critical case. 
We also study the conditional limiting distribution of the population 
outside the forefront.

Branching Brownian motions are a stochastic model 
describing the evolution of Brownian particles with reproduction.  
It is natural and interesting to investigate the interaction between 
the randomness of the reproduction and that of the particle motions. 

For instance, 
the spatial asymptotic distribution of particles is 
an expression of such interaction and characterized in terms of the principal eigenvalue and eigenfunction 
of the Schr\"odinger type 
operator associated with the branching structure  
(see, e.g., \cite{CRY17,CS07,EHK10,W67} for more general branching Markov processes). 

Here we are concerned with the {\it maximal displacement}, 
which is the trajectory of the maximal Euclidean norm of particles. 
Even though each particle obeys the law of the iterated logarithm, 
the maximal displacement may grow faster than a single Brownian particle 
because of the population growth. 
Let $L_t$ be the maximal displacement at time $t$. 
For simplicity, we assume the binary branching and that 
the splitting time is exponentially distributed with rate $1$.
We also assume that the initial state is a single particle at the origin 
and denote by ${\mathbf P}$ the law of this process. 
Bramson \cite{B78} ($d=1$, see also \cite{R13} for a simplified proof) 
and Mallein \cite{M15} ($d\geq 2$) proved that 
$L_t$ is expressed  as
\begin{equation}\label{eq:intro-asymp-u}
L_t=\sqrt{2}t+\frac{d-4}{2\sqrt{2}}\log t+Y_t,
\end{equation}
where $\{Y_t\}_{t\geq 0}$ is a real valued stochastic process 
such that $\lim_{k\rightarrow\infty}\sup_{t\geq 0}{\mathbf P}(|Y_t|\geq k)=0$.
This result says that $L_t$ grows linearly and the second order depends on the spatial dimension $d$. 
We note that Kyprianou \cite{K05} already obtained the linear growth rate of $L_t$ for $d\geq 2$. 
See also \cite{SBM20} for the weak convergence of a family of direction-wise derivative martingales 
which would be related to the direction of the extremal particle.  
For $d=1$, Chauvin and Rouault \cite{CR88, CR90} further determined 
the decay rate of 
the tail probability ${\mathbf P}(L_t>\delta t)$ as $t\rightarrow\infty$ for $\delta\geq \sqrt{2}$. 
For $\delta>\sqrt{2}$ especially, this probability is asymptotically equivalent 
to the expected population on the set $\{y\in {\mathbb R}\mid |y|>\delta t\}$ as $t\rightarrow\infty$, 
which is similar to the subcritical Galton-Watson process 
(see, e.g., \cite[Chapter I-8]{AN04} and \cite[Section 5]{R00}). 
We should mention that Bramson \cite{B78} and Chauvin and Rouault \cite{CR88,CR90} 
discussed the asymptotic properties of the rightmost particle for $d=1$, 
but their results immediately yield the corresponding ones for the maximal norm.  

Our purpose in this paper is to study  the maximal displacement 
for a {\it spatially inhomogeneous model}.  
Namely, the offspring distribution ${\mathbf p}$ is state dependent 
and the splitting time distribution is given 
by a Kato class measure $\mu$ on ${\mathbb R}^d$:
the splitting time of each particle is proportional to the size 
of $\mu$ along the trajectory 
(see Subsections \ref{subsect:kato-fk} and \ref{subsect:branching} for details) 
and particles reproduce only on the support of $\mu$.  
Here we assume that $\mu$ has compact support, that is, particles can not reproduce 
outside the compact set. 
Let $Q(x)$ be the expected offspring number at $x\in{\mathbb R}^d$. 
We can then regard the measure $\nu({\rm d}x):=(Q(x)-1)\mu({\rm d}x)$ 
as the branching intensity. 
Let $\lambda$ be the bottom of the spectrum of the Schr\"odinger type operator 
${\cal H}^{\nu}:=-\Delta/2-\nu$. 
We also assume that $\lambda<0$, that is, 
$\lambda$ is the principal eigenvalue of ${\cal H}^{\nu}$ 
and the branching intensity is strong enough. 
Under some Kato class and the second moment conditions on ${\mathbf p}$ 
(see Assumption \ref{assum:b} for details), 
we have the following assertions.  
\begin{enumerate}
\item[(i)] Let $\{Y_t\}_{t\geq 0}$ be a real valued stochastic process defined by 
\begin{equation*}
L_t=\sqrt{\frac{-\lambda}{2}}t+\frac{d-1}{2\sqrt{-2\lambda}}\log t + Y_t.
\end{equation*}
Then the conditional distribution of $\{Y_t\}_{t\geq 0}$ on the regular growth event 
is convergent to the Gumbel distribution shifted by some limiting martingale 
(see Theorem \ref{thm:lim-dist} and a subsequent comment for details). 
\item[(ii)] For any $\delta \in (\sqrt{-\lambda/2} , \sqrt{-2\lambda})$, 
there exists some explicit positive constant $C_1$ such that
\begin{equation*}
{\mathbf P}(L_t>\delta t)\sim 
C_1e^{(-\lambda-\sqrt{-2\lambda} \delta) t}t^{(d-1)/2} \ (t\rightarrow\infty).
\end{equation*}
There exists some explicit positive constant $C_2$ such that for any $\gamma\geq d-1$, 
\begin{equation*}
{\mathbf P}\left(L_t>\sqrt{\frac{-\lambda}{2}}t+\frac{\gamma}{2\sqrt{-2\lambda}}\log t+b(t)\right)
\sim C_2e^{-\sqrt{-2\lambda}b(t)}t^{(d-1-\gamma)/2} \ (t\rightarrow\infty).
\end{equation*}
Here $b(t)$ is a positive function on $(0,\infty)$ such that $b(t)=o(\log t) \ (t\rightarrow\infty)$, 
and $b(t)\rightarrow\infty \ (t\rightarrow\infty)$ if $\gamma=d-1$ 
(see Theorem \ref{thm:equiv}, \eqref{eq:tail-asymp} and \eqref{eq:tail-asymp-1} and for details). 
\end{enumerate}
For $d=1$, these assertions remain true for the trajectory of the rightmost particle 
(see Remark \ref{rem:dirac}). 

As we see from (i), 
the second order of $L_t$ is logarithmic and dependent of the spatial dimension $d$ 
as \eqref{eq:intro-asymp-u} for the spatially uniform model. 
The first order of $L_t$ was already obtained by \cite{BH14,E84,S18}. 
We also determine in (ii) the asymptotic behavior of the tail probability of $L_t$ 
for the subcritical and critical phases with respect to the linear growth rate.  
As a corollary of Theorem \ref{thm:equiv}, 
we determine the conditional limiting distribution 
of the population outside the forefront (Corollary \ref{thm:yaglom}). 
Note that if $\delta\geq \sqrt{-2\lambda}$ and 
non-extinction occurs, 
then we know by \cite{S19} that the effect of the reproduction does not appear  
in the behavior of ${\mathbf P}(L_t>\delta t)$ 
because particles can not reproduce outside a compact set 
and $\delta$ is large enough 
(see \eqref{eq:no-effect} for details).

Lalley and Sellke \cite{LS88} (see also \cite{LS89} for further results) proved the assertion in (i)
for $d=1$ provided that $\mu$ is finite and has a continuous density function 
with respect to the Lebesgue measure; 
however, no restriction is imposed on the support of $\mu$.
They compared the forefront of the branching Brownian motion 
with that of a Brownian particle system driven by a time-space Poisson point process 
(\cite[Section 3]{LS88}). 
They used the absolute continuity condition on $\mu$  
to express the intensity measure of the birth points of branching Brownian particles 
(\cite[(5.7)]{LS88}). 
Bocharov and Harris \cite{BH14, BH16} 
(see also \cite{BW19, WZ20} for related results) also obtained the assertion in (i)
and the exponential order of ${\mathbf P}(L_t>\delta t)$  
for the so called catalytic branching Brownian motion 
in which $d=1$ and $\mu$ is a multiple of the Dirac measure at the origin. 
Their approach is based on the moment calculus of the population.  
In fact, they compute the moments by using the explicit form of the joint distribution 
of the Brownian motion and its local time established by \cite{KS84}.  
Under the same setting as in the present  paper, we obtained in \cite[Theorem 3.7]{S19} 
a partial result on the polynomial order of the tail probability of $L_t$. 
To do so, we established its Feynman-Kac expression 
and used a less sharp estimate of the Feynman-Kac semigroups. 
This approach is similar to that of \cite{CR88,CR90} for the spatially homogeneous model,
for which the Feynman-Kac expression was obtained by McKean \cite{Mc75,Mc76}.

On the other hand, Carmona and Hu \cite{CH14} and Bulinskaya \cite{B20} studied 
the limiting distribution of the maximal displacement 
for the catalytic branching random walk on the integer lattice. 
They can utilize the renewal theory as the state space is discrete and 
particles reproduce only on the finite number of points.

Here we develop the moment calculus of the population as in \cite{BH16}. 
As we see from Lemma \ref{lem:moment} below, 
the population moments are expressed in terms of 
the Feynman-Kac semigroup associated with ${\mathbf p}$ and $\mu$. 
An important step is to reveal the precise and locally uniform long time behavior 
of the Feynman-Kac semigroup (\eqref{eq:both} and \eqref{eq:both-1}). 
This provides the detailed asymptotics of the first and second population moments. 
Combining this with the Chebyshev and Paley-Zygmund inequalities (see \eqref{eq:moment}), 
we obtain the limiting behaviors of the tail probabilities of $L_t$ 
local uniformly with respect to the initial point. 
We can then follow the argument of \cite{CH14} and \cite{B20} 
to establish Theorem \ref{thm:lim-dist}. 
Theorem \ref{thm:equiv} is also proved 
by using the asymptotics of the population moments.

To achieve the step as mentioned above, 
we use the Poincar\'e inequality \eqref{eq:poincare} 
for the Feynman-Kac semigroup, 
which was already applied in \cite{CRY17,CS07} to the limit theorem 
for branching Markov processes. 
We emphasize that the moment calculus of the population is feasible 
by using the principal eigenvalue and spectral gap of ${\cal H}^{\nu}$.  
The price is to impose Kato class and compact support conditions on ${\mathbf p}$ and $\mu$; 
the former condition guarantees the existence of the spectral gap for ${\cal H}^{\nu}$ 
and thus the Poincar\'e inequality, 
and the latter allows us to utilize the locally uniform long time behavior of the Feynman-Kac semigroup.   
We also need the existence of the second moment of ${\mathbf p}$ 
to derive the lower bound of the tail probability of $L_t$ 
by using \eqref{eq:moment}.

The rest of this paper is organized as follows. 
In Section \ref{sect:pre-results}, we first introduce the model of branching Brownian motions. 
We then present our results and their applications to the concrete models.   
In Section \ref{sect:fk-est}, we derive the long time asymptotic properties of Feynman-Kac semigroups. 
The subsequent sections are devoted to the proofs of the results presented in Section \ref{sect:pre-results}. 
In Appendix \ref{appendix:int}, we give a part on the elementary calculation in Section \ref{sect:fk-est}.

Throughout this paper, the letters $c$ and $C$ (with subscript) 
denote finite positive constants which may vary from place to place. 
For positive functions $f(t)$ and $g(t)$ on $(0,\infty)$, we write 
$f(t)\asymp g(t) \ (t\rightarrow\infty)$ if there exist positive constants $T$, $c_1$ and $c_2$ such that 
$c_1g(t)\leq f(t)\leq c_2g(t)$ for all $t\geq T$.  
We also write $f(t)\sim  g(t) \ (t\rightarrow \infty)$ if $f(t)/g(t)\rightarrow 1 \ (t\rightarrow\infty)$.

\section{Preliminaries and results}\label{sect:pre-results}
\subsection{Kato class measures and Feynman-Kac semigroups}\label{subsect:kato-fk}
Let 
${\mathbf M}=(\Omega, {\cal F}, \{{\cal F}_t\}_{t\geq 0}, 
\{B_t\}_{t\geq 0}, \{P_x\}_{x\in {\mathbb R}^d}, \{\theta_t\}_{t\geq 0})$ 
be the Brownian motion on ${\mathbb R}^d$, 
where $\{{\cal F}_t\}_{t\geq 0}$  is the minimal augmented admissible filtration 
and $\{\theta_t\}_{t\geq 0}$  is the time shift operator of paths 
such that $B_s\circ\theta_t=B_{s+t}$ identically for $s,t\geq 0$. 
Let $p_t(x,y)$ be the transition density function of ${\mathbf M}$ given by 
$$p_t(x,y)=\frac{1}{(2\pi t)^{d/2}}\exp\left(-\frac{|x-y|^2}{2t}\right),
\ (t,x,y)\in (0,\infty)\times {\mathbb R}^d\times {\mathbb R}^d.$$
For $\alpha>0$, let $G_{\alpha}(x,y)$ be 
the $\alpha$-resolvent of ${\mathbf M}$ given by 
$$G_{\alpha}(x,y)=\int_0^{\infty}e^{-\alpha t}p_t(x,y)\,{\rm d}t.$$
Then 
\begin{equation}\label{eq:resolvent}
G_{\alpha}(x,y)\sim \frac{1}{\sqrt{2\alpha}}\left(\frac{\sqrt{2\alpha}}{2\pi |x-y|}\right)^{(d-1)/2}
e^{-\sqrt{2\alpha}|x-y|} \ (|x-y|\rightarrow\infty)
\end{equation}
(see, e.g., \cite[(2.1)]{S18} and references therein). 
For $d\geq 3$, let $G(x,y)$ be the Green function of ${\mathbf M}$ defined by 
$$G(x,y)=\int_0^{\infty}p_t(x,y)\,{\rm d}t$$
and $G_0(x,y):=G(x,y)$.

\begin{defn} 
\begin{itemize}
\item[{\rm (i)}] Let $\mu$ be a positive Radon measure on ${\mathbb R}^d$. 
Then $\mu$ belongs to the Kato class {\rm (}$\mu\in {\cal K}$ in notation{\rm )} if 
$$\lim_{\alpha\rightarrow\infty}
\sup_{x\in {\mathbb R}^d}\int_{{\mathbb R}^d}G_{\alpha}(x,y)\,\mu({\rm d}y)=0.$$
\item[{\rm (ii)}] For $\beta>0$, a measure $\mu\in {\cal K}$ is $\beta$-Green tight 
{\rm (}$\mu\in {\cal K}_{\infty}(\beta)$ in notation{\rm )} if
$$\lim_{R\rightarrow\infty}\sup_{x\in {\mathbb R}^d}\int_{{|y|\geq R}}G_{\beta}(x,y)\,\mu({\rm d}y)=0.$$ 
When $d\geq 3$, $\mu\in {\cal K}$ belongs to ${\cal K}_{\infty}(0)$ if 
the equality above is valid for $\beta=0$.
\end{itemize}
\end{defn}

We know by \cite{T08} that for any $\beta>0$, ${\cal K}_{\infty}(\beta)$ is independent of $\beta$. 
Any Kato class measure with compact support is $1$-Green tight by definition. 

For $\mu\in{\cal K}$, 
let $A_t^{\mu}$ be a positive continuous additive functional 
in the Revuz correspondence to $\mu$ (see, e.g., \cite[p.401]{FOT11} for details). 
For instance, if $\mu({\rm d}x)=V(x)\,{\rm d}x$ for some nonnegative function $V(x)\geq 0$, 
then $A_t^{\mu}=\int_0^t V(B_s)\,{\rm d}s$. 

Let $\nu$ be a signed measure on ${\mathbb R}^d$ such that 
$\nu=\nu^{+}-\nu^{-}$ for some $\nu^+,\nu^{-}\in {\cal K}$ 
and let $A_t^{\nu}=A_t^{\nu^+}-A_t^{\nu^-}$.
We define the Feynman-Kac semigroup $\{p_t^{\nu}\}_{t>0}$ by 
$$p_t^{\nu}f(x):=E_x\left[e^{A_t^{\nu}}f(B_t)\right], 
\  f\in {\cal B}_b({\mathbb R}^d)\cap L^2({\mathbb R}^d),$$
where ${\cal B}_b({\mathbb R}^d)$ stands for the totality of bounded Borel measurable functions on ${\mathbb R}^d$. 
Then $\{p_t^{\nu}\}_{t>0}$ forms 
a strongly continuous symmetric semigroup on $L^2({\mathbb R}^d)$ 
and the associated $L^2$-generator is formally written 
as the Schr\"odinger type operator ${\cal H}^{\nu}=-\Delta/2-\nu$. 
We can further extend  $\{p_t^{\nu}\}_{t>0}$ to $L^p({\mathbb R}^d)$ 
for any $p\in [1,\infty]$ (\cite[Theorem 6.1 (i)]{ABM91}). 
We use the same notation $\{p_t^{\nu}\}_{t>0}$ for the extended semigroups. 

By \cite[Theorems 7.1 and 8.1]{ABM91}, 
there exists a jointly continuous integral kernel $p_t^{\nu}(x,y)$ 
on $(0,\infty)\times {\mathbb R}^d\times {\mathbb R}^d$ 
such that 
\begin{equation*}
p_t^{\nu}f(x)=\int_{{\mathbb R}^d}p_t^{\nu}(x,y)f(y)\,{\rm d}y, 
\  f\in {\cal B}_b({\mathbb R}^d)
\end{equation*}
and 
\begin{equation}\label{eq:heat-comp}
p_t^{\nu}(x,y)\leq c_1 p_{c_2 t}(x,y),  \ (t,x,y)\in (0,1]\times {\mathbb R}^d\times {\mathbb R}^d
\end{equation}
for some $c_1>0$ and $c_2>0$. 
Moreover, there exists $\beta(\nu)>0$ by \cite[Theorem 6.1]{ABM91} 
such that for any $\alpha>\beta(\nu)$, 
we can associate the resolvent $\{G_{\alpha}^{\nu}\}_{\alpha>\beta(\nu)}$ 
given by   
\begin{equation*}
G_{\alpha}^{\nu}f(x)=\int_0^{\infty}e^{-\alpha t}p_t^{\nu}f(x)\,{\rm d}t
=E_x\left[\int_0^{\infty} e^{-\alpha t+A_t^{\nu}}f(B_t)\,{\rm d}t\right].
\end{equation*}

Let $\nu^+,\nu^{-}\in {\cal K}_{\infty}(1)$ and $\nu=\nu^{+}-\nu^{-}$.   
Denote by $\sigma({\cal H}^{\nu})$ 
the spectrum for ${\cal H}^{\nu}$ 
and $\lambda(\nu):=\inf\sigma({\cal H}^{\nu})$. 
Then 
\begin{equation*}
\lambda(\nu)=
\inf\left\{\frac{1}{2}\int_{{\mathbb R}^d}|\nabla u|^2\,{\rm d}x
-\int_{{\mathbb R}^d}u^2\,{\rm d}\nu \mid 
u\in C_0^{\infty}({\mathbb R}^d), \int_{{\mathbb R}^d}u^2\,{\rm d}x=1\right\},
\end{equation*}
where $C_0^{\infty}({\mathbb R}^d)$ stands for 
the totality of smooth functions on ${\mathbb R}^d$ with compact support. 
If $\lambda(\nu)<0$, then $\lambda(\nu)$ is the principal eigenvalue of ${\cal H}^{\nu}$ 
(\cite[Lemma 4.3]{T03} or \cite[Theorem 2.8]{T08})
and the corresponding eigenfunction has a bounded, 
continuous and strictly positive version (\cite[Section 4]{T08}). 
We write $h$ for this version with $L^2$-normalization $\|h\|_{L^2({\mathbb R}^d)}=1$. 
Hence for any $x\in {\mathbb R}^d$ and $t>0$,
\begin{equation}\label{eq:eigen-eq}
p_t^{\nu}h(x)=E_x\left[e^{A_t^{\nu}}h(B_t)\right]=e^{-\lambda t}h(x).
\end{equation}
If $\nu^{+}$ and $\nu^{-}$ are in addition compactly supported in ${\mathbb R}^d$, 
then by the proof of \cite[Theorem 5.2]{T08} or \cite[Appendix A.1]{S19}, 
there exist positive constants $c_1$ and $c_2$ such that  
\begin{equation}\label{eq:est-eigen}
\frac{c_1e^{-\sqrt{-2\lambda(\nu)}|x|}}{|x|^{(d-1)/2}}\leq h(x)
\leq \frac{c_2e^{-\sqrt{-2\lambda(\nu)}|x|}}{|x|^{(d-1)/2}} 
\quad (|x|\geq 1).
\end{equation}

Let $\lambda_2(\nu)$ be the second bottom of the spectrum for ${\cal H}^{\nu}$:
\begin{equation*}
\lambda_2(\nu)
:=\inf\left\{\frac{1}{2}\int_{{\mathbb R}^d}|\nabla u|^2\,{\rm d}x-\int_{{\mathbb R}^d}u^2\,{\rm d}\nu 
\mid u\in C_0^{\infty}({\mathbb R}^d), 
\int_{{\mathbb R}^d}u^2\,{\rm d}x=1, \int_{{\mathbb R}^d}uh\,{\rm d}x=0\right\}.
\end{equation*}
If $\lambda(\nu)<0$, then $\lambda(\nu)<\lambda_2(\nu)\leq 0$ 
because the essential spectrum of ${\cal H}^{\nu}$ is the interval $[0,\infty)$ 
by \cite[Theorem 3.1]{BEKS94} or \cite[Lemma 3.1]{Be04}, 
Moreover, there exists $C>0$ such that 
for any $t\geq 1/2$, $x\in {\mathbb R}^d$ 
and $\varphi\in L^2({\mathbb R}^d)$ with $\int_{{\mathbb R}^d}\varphi(y) h(y)\,{\rm d}y=0$, 
\begin{equation}\label{eq:poincare}
\left|p_t^{\nu}\varphi(x)\right|\leq Ce^{-\lambda_2(\nu) t}\|\varphi\|_{L^2({\mathbb R}^d)}
\end{equation}
(see, e.g., \cite[Subsection 2.1]{CS07}).

\subsection{Branching Brownian motions}\label{subsect:branching}
In this subsection, we introduce a model of branching Brownian motions  
(see \cite{INW68-1,INW68-2,INW69} and \cite{S18,S19} for details). 
Let $\mu\in {\cal K}$ and let ${\mathbf p}=\{p_n(x)\}_{n=0}^{\infty}$ be 
a probability function on ${\mathbb R}^d$. 
A  Brownian particle with initial site $x\in {\mathbb R}^d$ 
moves until random time $Z$, whose distribution is given by 
\begin{equation*}
P_x(Z>t \mid {\cal F}_{\infty})=e^{-A_t^{\mu}} \quad (t>0).
\end{equation*}
At time $Z$, this particle disappears leaving no offspring with probability $p_0(B_{Z_-})$. 
Otherwise, this particle splits into $n \ (\geq 1)$ particles with probability $p_n(B_{Z-})$,  
and then these new particles with common initial site $B_{Z-}$ 
repeat the same procedure independently.
When there are $n \ (\geq 2)$ particles at time $0$, 
these particles evolve independently in a similar way.   
A particle system constructed in this way is called a branching Brownian motion 
with branching rate $\mu$ and branching mechanism ${\mathbf p}$. 
We use the notation $\overline{\bf M}$ to indicate this particle system. 

For $n\geq 2$, we introduce an equivalence relation on $({\mathbb R}^d)^n$ 
by ignoring the order of components, 
and denote by $({\mathbb R}^d)^{(n)}$ the associated equivalence class. 
We also let $({\mathbb R}^d)^{(1)}={\mathbb R}^d$. 
Let $Z_t$ be the total number of particles at time $t$. 
If $Z_t=n$ for some $n\geq 1$, 
then we let 
${\mathbf B}_t=({\mathbf B}_t^{1},\dots, {\mathbf B}_t^n)\in ({\mathbb R}^d)^{(n)}$ 
denote the location of particles alive at time $t$. 
For ${\mathbf x}=(x^1,\dots,x^{n})\in ({\mathbb R}^d)^{(n)}$, 
let ${\mathbf P}_{\mathbf x}$ be the law of $\overline{\bf M}$  
under the condition  that the initial $n$ particles are located at 
$x^1,\dots, x^n\in {\mathbb R}^d$.

Let $e_0=\inf\{t>0 \mid Z_t=0\}$ be the extinction time of $\overline{\mathbf M}$ 
and $u_e(x)={\mathbf P}_x(e_0<\infty)$. 
We say that $\overline{\mathbf M}$ becomes extinct if $u_e\equiv 1$.
Note that on the event $\{e_0<\infty\}$, we have $Z_t=0$ for any $t\geq e_0$.

Define for $f\in {\cal B}_b({\mathbb R}^d)$,
\begin{equation*}
Z_t(f)=
\begin{cases}
\displaystyle \sum_{k=1}^{Z_t}f({\mathbf B}_t^k) & (t<e_0),\\ 
0 & (t\geq e_0).
\end{cases}
\end{equation*}
Let 
\begin{equation}\label{eq:exp-offspring}
Q(x)=\sum_{n=1}^{\infty}np_n(x), \quad R(x)=\sum_{n=2}^{\infty}n(n-1)p_n(x).
\end{equation}
The following lemma is sometimes called the many-to-one or many-to-two lemma 
(see, e.g., \cite{HR17}). 

\begin{lem}\label{lem:moment} 
{\rm (\cite[Lemma 3.3]{S08}).}
Let $\mu\in {\cal K}$. 
\begin{enumerate}
\item If the measure 
$$
\nu_Q({\rm d}x):=Q(x)\mu({\rm d}x)
$$
also belongs to the Kato class, 
then for any $f\in {\cal B}_b({\mathbb R}^d)$,
\begin{equation}\label{eq:br-fk}
{\mathbf E}_x\left[Z_t(f)\right]
=E_x\left[e^{A_t^{(Q-1)\mu}}f(B_t)\right].
\end{equation}
\item If the measure 
$$
\nu_R({\rm d}x):=R(x)\mu({\rm d}x)
$$
also belongs to the Kato class, then for any $f, g\in {\cal B}_b({\mathbb R}^d)$,
\begin{equation}\label{eq:br-fk-1}
\begin{split}
&{\mathbf E}_x\left[Z_t(f)Z_t(g)\right]
=E_x\left[e^{A_t^{(Q-1)\mu}}f(B_t)g(B_t)\right]\\
&+E_x\left[\int_0^t e^{A_s^{(Q-1)\mu}}E_{B_s}\left[e^{A_{t-s}^{(Q-1)\mu}}f(B_{t-s})\right]
E_{B_s}\left[e^{A_{t-s}^{(Q-1)\mu}}g(B_{t-s})\right]\,{\rm d}A_s^{\nu _R}\right].
\end{split}
\end{equation}
\end{enumerate}
\end{lem}

We note that \eqref{eq:br-fk} and \eqref{eq:br-fk-1} were proved in \cite[Lemma 3.3]{S08} 
under the condition that $Q(x)$ and $R(x)$ are bounded on ${\mathbb R}^d$. 
However, that proof still works under the weak conditions as in Lemma \ref{lem:moment}. 
Since $Q(x)\leq R(x)+1$ and $\mu \in {\cal K}$,  
we see that $\nu_R\in {\cal K}$ implies $\nu_Q\in {\cal K}$.

\subsection{Results}
Let $\mu\in {\cal K}$ and 
let ${\mathbf p}$ be a probability function on ${\mathbb R}^d$.
Let $\overline{\mathbf M}=(\{{\mathbf B}_t\}_{t\geq 0}, \{{\mathbf P}_{\mathbf x}\}_{{\mathbf x}\in {\mathbf X}})$ 
be the branching Brownian motion on ${\mathbf X}$ with branching rate $\mu$ and branching mechanism ${\mathbf p}$. 
We make the next assumption on $\mu$ and ${\mathbf p}$.

\begin{assum}\label{assum:b}\rm 
\begin{enumerate}
\item $\mu$ is a Kato class measure with compact support in ${\mathbb R}^d$. 
\item $\nu_R({\rm d}x)\in {\cal K}$. 
\item Let $\lambda:=\lambda((Q-1)\mu)$. Then $\lambda<0$.
\end{enumerate}
\end{assum}

As we see from Subsection \ref{subsect:kato-fk}, 
Assumption \ref{assum:b} implies that  $\lambda$ is the principal eigenvalue 
of the operator ${\cal H}^{(Q-1)\mu}$ 
and the corresponding $L^2$-normalized eigenfunction has a version $h$  
which is bounded, continuous and strictly positive on ${\mathbb R}^d$.
We also know by \cite[Lemmas 4.1 and 4.2]{T03} that if  $d=1,2$, $p_0\equiv 0$  and $(Q-1)\mu$ is non-trivial,  
then Assumption \ref{assum:b} (iii) follows from (i) and (ii).

Let
\begin{equation*}
M_t=e^{\lambda t}Z_t(h).
\end{equation*}
Since $\nu_R$ is a Kato class measure with compact support, 
$M_t$ is a square integrable nonnegative ${\mathbf P}_x$-martingale 
by the same argument as in \cite[Lemma 3.4]{S08}. 
In fact, we have 
\begin{equation}\label{eq:mtg-moment-1}
{\mathbf E}_x\left[M_t\right]=h(x)
\end{equation}
and
\begin{equation}\label{eq:mtg-moment-2}
{\mathbf E}_x\left[M_t^2\right]
=e^{2\lambda t}E_x\left[e^{A_t^{(Q-1)\mu}}h(B_t)^2\right]
+E_x\left[\int_0^te^{2\lambda s+A_s^{(Q-1)\mu}}h(B_s)^2\,{\rm d}A_s^{\nu_R}\right].
\end{equation}
Note that the limit $M_{\infty}:=\lim_{t\rightarrow\infty}M_t$ exists ${\mathbf P}_x$-a.s.\ 
and ${\mathbf P}_x(M_{\infty}>0)>0$. 
Since \cite[Lemma 3.9]{S18} and \cite[Theorem 3.2 (ii)]{S19} imply that 
\begin{equation*}
\lim_{t\rightarrow\infty}\frac{1}{t}\log Z_t
=\lim_{t\rightarrow\infty}\frac{1}{t}\log{\mathbf E}_x[Z_t]=-\lambda,  
\ \text{${\mathbf P}_x(\cdot\mid M_{\infty}>0)$-a.s.},
\end{equation*}
we can regard the event $\{M_{\infty}>0\}$ as the regular growth event. 

Let $L_t$ be the maximum of the Euclidean norms of particles alive at time $t$:
\begin{equation*}
L_t=\begin{cases}
\max_{1\leq k\leq Z_t}|{\mathbf B}_t^k| & (t<e_0),\\
0 & (t\geq e_0).
\end{cases}\end{equation*} 
Then by \cite[Corollary 3.3]{S19}, 
\begin{equation}\label{eq:maximal}
\lim_{t\rightarrow\infty}\frac{L_t}{t}=\sqrt{\frac{-\lambda}{2}}, \quad \text{${\mathbf P}_x(\cdot \mid M_{\infty}>0)$-a.s.}
\end{equation}
In particular, if $d=1, 2$, then this equality is valid ${\mathbf P}_x(\cdot \mid e_0=\infty)$-a.s.

We are now in a position to state our results in this paper.  
The first result provides the second growth order  of $L_t$ as $t\rightarrow\infty$. 
For $\kappa\in {\mathbb R}$, let 
\begin{equation}\label{eq:rad-br}
R_1^{(\kappa)}(t)=\sqrt{\frac{-\lambda}{2}}t+\frac{d-1}{2\sqrt{-2\lambda}}\log t+\kappa.
\end{equation}

\begin{thm}\label{thm:lim-dist}
Under Assumption {\rm \ref{assum:b}}, 
the next equality holds for each $\kappa\in {\mathbb R}${\rm :}
\begin{equation*}
\lim_{t\rightarrow\infty}{\mathbf P}_x(L_t>R_1^{(\kappa)}(t))
={\mathbf E}_x\left[1-\exp\left(-c_*e^{-\sqrt{-2\lambda}\kappa}M_{\infty}\right)\right].
\end{equation*}
Here $c_*$ is a positive constant, 
which will be given in {\rm \eqref{eq:const-1}} below, with $\nu=(Q-1)\mu$.
\end{thm}

For $d=1$, the same result was proved by \cite{BH16} and \cite{LS88}, 
but Theorem \ref{thm:lim-dist} allows the singularity of the branching rate measure 
even for $d\geq 2$. 
For catalytic branching random walks, this result was established by 
Carmona and Hu \cite{CH14} and Bulinskaya \cite{B20}.

Here we make a comment on Theorem \ref{thm:lim-dist}.
Let $\{Y_t\}_{t\geq 0}$ be a real valued stochastic process defined by 
\begin{equation}\label{eq:asymp-l}
L_t=\sqrt{\frac{-\lambda}{2}}t+\frac{d-1}{2\sqrt{-2\lambda}}\log t + Y_t.
\end{equation}
Then 
$\{e_0<\infty\}\subset \{M_{\infty}=0\}$ and 
\begin{equation*}
{\mathbf P}_x(L_t>R_1^{(\kappa)}(t), e_0<\infty)
={\mathbf P}_x(L_t>R_1^{(\kappa)}(t), t<e_0<\infty)\leq {\mathbf P}_x(t<e_0<\infty)\rightarrow 0
\end{equation*}
as $t\rightarrow\infty$. 
We also know by \cite[Proposition A.5]{S19} that 
\begin{itemize}
\item For $d=1,2$, $\{M_{\infty}>0\}=\{e_0=\infty\}$, ${\mathbf P}_x$-a.s.; 
\item For $d\geq 3$, ${\mathbf P}_x(\{e_0=\infty\}\cap \{M_{\infty}=0\})>0$.
\end{itemize}
For $d\geq 3$,  
$L_t$ satisfies the law of the iterated logarithm 
on the event $\{e_0=\infty\}\cap \{M_{\infty}=0\}$ (\cite[Remark 3.5]{S19}) 
so that  
\begin{equation*}
{\mathbf P}_x(L_t>R_1^{(\kappa)}(t), e_0=\infty, M_{\infty}=0)\rightarrow 0 \ (t\rightarrow\infty).
\end{equation*}
Hence for any $d\geq 1$, 
\begin{equation*}
\lim_{t\rightarrow\infty}{\mathbf P}_x(L_t>R_1^{(\kappa)}(t)\mid M_{\infty}>0)
={\mathbf E}_x\left[1-\exp\left(-c_*e^{-\sqrt{-2\lambda}\kappa}M_{\infty}\right) \mid M_{\infty}>0\right],
\end{equation*}
that is, 
\begin{equation*}
\lim_{t\rightarrow\infty}{\mathbf P}_x(Y_t\leq \kappa \mid M_{\infty}>0)
={\mathbf E}_x\left[\exp\left(-c_*e^{-\sqrt{-2\lambda}\kappa}M_{\infty}\right) \mid M_{\infty}>0\right].
\end{equation*} 
This means that $\{Y_t\}_{t\geq 0}$ 
under ${\mathbf P}_x(\cdot\mid M_{\infty}>0)$ converges in distribution 
to a Gumbel random variable, 
which is independent of $M_{\infty}$, with random shift $(-2\lambda)^{-1/2}\log M_{\infty}$.

In \cite[Theorem 3.7 (ii)]{S19}, we obtained bounds for the polynomial order 
of the tail probability of $L_t$. 
As its refinement, the next result determines the precise asymptotic behavior of the tail probability.
Let $a(t)$ be  a function on $(0,\infty)$ such that 
$a(t)=o(t)$ as $t\rightarrow\infty$ and define for $\delta\in (\sqrt{-\lambda/2},\sqrt{-2\lambda})$, 
\begin{equation}\label{eq:rad-br-1}
R_2(t)=\delta t+a(t).
\end{equation}
Let $b(t)$  be a function on $(0,\infty)$ such that 
$b(t)=o(\log t)$ as $t\rightarrow\infty$.
For $\gamma\geq d-1$, we define  
\begin{equation}\label{eq:rad-br-2}
R_3(t)=\sqrt{\frac{-\lambda}{2}} t+\frac{\gamma}{2\sqrt{-2\lambda}}\log t+b(t).
\end{equation}
If $\gamma=d-1$, then we assume in addition that $b(t)\rightarrow\infty$ as $t\rightarrow\infty$. 
For $R>0$, let $Z_t^R$ be the total number of particles on $\{x\in {\mathbb R}^d \mid |x|>R\}$ at time $t$. 
\begin{thm}\label{thm:equiv}
Let $K$ be a compact set in ${\mathbb R}^d$ and 
let  Assumption {\rm \ref{assum:b}} hold.  
Then for any $\delta\in (\sqrt{-\lambda/2},\sqrt{-2\lambda})$, 
\begin{equation*}
\lim_{t\rightarrow\infty}\inf_{x\in K}
\frac{{\mathbf P}_x(L_t> R_2(t))}
{{\mathbf E}_x\left[Z_t^{R_2(t)}\right]}
=\lim_{t\rightarrow\infty}\sup_{x\in K}
\frac{{\mathbf P}_x(L_t>R_2(t))}
{{\mathbf E}_x\left[Z_t^{R_2(t)}\right]}=1.
\end{equation*}
These equalities are still true if $R_2(t)$ is replaced by $R_3(t)$ 
with any $\gamma\geq d-1$.
\end{thm}

On account of \eqref{eq:br-fk} with \eqref{eq:asymp-eta} and \eqref{eq:both-1} below, 
Theorem \ref{thm:equiv} asserts that local uniformly in $x\in {\mathbb R}^d$, 
for any $\delta\in (\sqrt{-\lambda/2},\sqrt{-2\lambda})$,
\begin{equation}\label{eq:tail-asymp}
{\mathbf P}_x(L_t>R_2(t))\sim {\mathbf E}_x\left[Z_t^{R_2(t)}\right]
\sim c_d \delta^{(d-1)/2}h(x)e^{(-\lambda-\sqrt{-2\lambda} \delta) t}e^{-\sqrt{-2\lambda}a(t)}t^{(d-1)/2} \ (t\rightarrow\infty)
\end{equation}
and for any $\gamma\geq d-1$,  
\begin{equation}\label{eq:tail-asymp-1}
{\mathbf P}_x(L_t>R_3(t))\sim {\mathbf E}_x\left[Z_t^{R_3(t)}\right]
\sim c_*
h(x)e^{-\sqrt{-2\lambda}b(t)} t^{(d-1-\gamma)/2} \ (t\rightarrow\infty). 
\end{equation}
Here $c_d$ is a positive constant with $\nu=(Q-1)\mu$, 
which will be given in \eqref{eq:const} below, 
and $c_*$ is the same constant as in Theorem \ref{thm:lim-dist}.

Note that if  $\delta\geq \sqrt{-2\lambda}$ and $p_0\equiv 0$, 
then by \cite[Remark 3.8 and (4.1)]{S19}, 
\begin{equation}\label{eq:no-effect}
{\mathbf P}_x(L_t>\delta t)\asymp {\mathbf E}_x[Z_t^{\delta t}]
\asymp P_x(|B_t|>\delta t) \asymp e^{-\delta^2t/2}t^{(d-2)/2} \ (t\rightarrow\infty).
\end{equation}
However, we do not know whether it is possible to refine this relation as in Theorem \ref{thm:equiv}.

As a corollary of Theorem \ref{thm:equiv}, 
we determine the conditional limiting distribution of the population 
outside the forefront. 
\begin{cor}\label{thm:yaglom}
Let $K$ be a compact set in ${\mathbb R}^d$ and let Assumption {\rm \ref{assum:b}} hold.  
Then for any $\delta\in (\sqrt{-\lambda/2},\sqrt{-2\lambda})$, 
\begin{equation*}
\begin{split}
\lim_{t\rightarrow\infty}\inf_{x\in K}{\mathbf P}_x\left(Z_t^{R_2(t)}=k \mid L_t>R_2(t)\right)
&=\lim_{t\rightarrow\infty}\sup_{x\in K}{\mathbf P}_x\left(Z_t^{R_2(t)}=k \mid L_t>R_2(t)\right)\\
&=\begin{cases} 
1 & (k=1), \\ 
0 & (k\geq 2).
\end{cases}
\end{split}
\end{equation*}
These equalities are still true if $R_2(t)$ is replaced by $R_3(t)$ 
with any $\gamma\geq d-1$.
\end{cor}

We note that particles can not reproduce outside the support of the branching rate measure $\mu$,  
which is compact by assumption. 
Because of this property, the conditional limiting distribution in 
Corollary \ref{thm:yaglom} 
is different from the so-called Yaglom type limit 
for spatially homogeneous branching Brownian motions 
as in \cite{CR88,CR90}.

\begin{rem}\label{rem:dirac}\rm 
Suppose that $d=1$. 
Let $R_t:=\max_{1\leq k\leq Z_t}{\mathbf B}_t^{k}$ be 
the position of the rightmost particle at time $t$. 
Then Theorem \ref{thm:lim-dist} remains true by replacing $L_t$ and $c_*$ 
with $R_t$ and $c_0$ in \eqref{eq:const-1-0} below, respectively. 
Theorem \ref{thm:equiv} and Corollary \ref{thm:yaglom} 
also remain true 
by replacing $L_t$ and  $Z_t^{R_i(t)}$, respectively, 
with $R_t$  and the number of particles on the interval $(R_i(t),\infty)$ at time $t$.
We also need to replace the constant $c_d$ 
in \eqref{eq:tail-asymp} and \eqref{eq:tail-asymp-1} with $c_0$. 
The proofs of these statements are almost identical with the original ones, 
but we use \eqref{eq:int-h} instead of \eqref{eq:int-h-0}. 
\end{rem}

\subsection{Examples}
In this subsection, we apply our results to some concrete models.

\begin{exam}\label{exam:dirac}\rm 
Suppose that $d=1$ and $p_0(x)+p_2(x)\equiv 1$. 
Then $Q(x)-1=2p_2(x)-1$. 
Let $\nu=(Q-1)\mu$ and $\lambda=\lambda(\nu)$. 
Then under Assumption \ref{assum:b}, 
Theorem \ref{thm:lim-dist} is true with 
\begin{equation*}
R_1^{(\kappa)}(t)=\sqrt{\frac{-\lambda}{2}}t+\kappa.
\end{equation*}

Let 
\begin{equation*}
C_0=\frac{1}{-2\lambda}
\int_{{\mathbb R}}\left(e^{\sqrt{-2\lambda}z}+e^{-\sqrt{-2\lambda}z}\right)h(z)\,\nu({\rm d}z).
\end{equation*}
Then by \eqref{eq:tail-asymp} and \eqref{eq:const} below, 
we get for any $\delta\in (\sqrt{-\lambda/2}, \sqrt{-2\lambda})$,
\begin{equation}\label{eq:tail-asymp-exam}
{\mathbf P}_x(L_t>\delta t)\sim 
C_0h(x)e^{(-\lambda-\sqrt{-2\lambda} \delta) t}
 \ (t\rightarrow\infty).
\end{equation}
By \eqref{eq:tail-asymp-1} and \eqref{eq:const} below, we have for any $\gamma>0$,
\begin{equation}\label{eq:tail-asymp-exam-1}
{\mathbf P}_x\left(L_t>\sqrt{\frac{-\lambda}{2}}t+\frac{\gamma}{2\sqrt{-2\lambda}}\log t\right)\\
\sim C_0h(x)t^{-\gamma/2}
 \ (t\rightarrow\infty)
\end{equation}
and 
\begin{equation}\label{eq:tail-asymp-exam-2}
{\mathbf P}_x\left(L_t>\sqrt{\frac{-\lambda}{2}}t+b(t)\right)\\
\sim C_0h(x)e^{-\sqrt{-2\lambda}b(t)}
 \ (t\rightarrow\infty),
\end{equation}
where $b(t)$ is a function on $(0,\infty)$ such that $b(t)=o(\log t)$ and $b(t)\rightarrow\infty$ as $t\rightarrow\infty$. 
By Remark \ref{rem:dirac}, 
we also obtain the exact decay order of the tail distribution of the rightmost particle similar to  
\eqref{eq:tail-asymp-exam}, \eqref{eq:tail-asymp-exam-1} and \eqref{eq:tail-asymp-exam-2}. 

\begin{enumerate}
\item Let $\delta_0$ be the Dirac measure at the origin and $\mu=\beta\delta_0$ for some $\beta>0$. 
Let $p=p_2(0)$ satisfy $1/2<p\leq 1$. 
Then $\lambda=-\{(2p-1)\beta\}^2/2$ (see, e.g., \cite[Example 4.4]{S08})
and thus 
\begin{equation*}
R_1^{(\kappa)}(t)=\frac{(2p-1)\beta}{2}t+\kappa.
\end{equation*}
When $p=1$, Theorem \ref{thm:lim-dist} was already proved in \cite{BH16}. 

By Lemma \ref{lem:int-eigen} (iii), 
\begin{equation*}
h(x)=(2p-1)\beta G_{\{(2p-1)\beta\}^2/2}(x,0)h(0).
\end{equation*}
Since $\|h\|_{L^2({\mathbb R})}=1$, we have $h(0)^2=(2p-1)\beta$ 
so that \eqref{eq:tail-asymp-exam-2} becomes 
\begin{equation*}
{\mathbf P}_x\left(L_t>\sqrt{\frac{-\lambda}{2}}t+b(t) \right)
\sim 2(2p-1)\beta G_{\{(2p-1)\beta\}^2/2}(x,0) e^{-(2p-1)\beta b(t)}
 \ (t\rightarrow\infty).
\end{equation*}
We can also rewrite \eqref{eq:tail-asymp-exam} and 
\eqref{eq:tail-asymp-exam-1} in a similar way.

\item For $a>0$, let $\mu=\delta_a+\delta_0$. 
Let $p=p_2(a)$ and $q=p_2(0)$ 
so that $(Q-1)\mu=(2p-1)\delta_0+(2q-1)\delta_a$. 
Assume that $p\geq q$ for simplicity. 
Then $\lambda<0$ if and only if one of the following conditions hold:
\begin{itemize}
\item $p>1/2$ and $q\geq 1/2$;
\item $p>1/2$, $q<1/2$ and 
\begin{equation*}
2p-1>\frac{1-2q}{1+2a(1-2q)}
\end{equation*}
\end{itemize}
(see \cite[Example 3.10]{S19} and references therein). 
Theorems \ref{thm:lim-dist}, \ref{thm:equiv} and Corollary \ref{thm:yaglom} 
are valid under 
either of these conditions. 
\end{enumerate} 
\end{exam}

\begin{exam}\rm 
Suppose that $d\geq 2$.  
For $R>0$, let $\delta_R$ be a surface measure on the sphere $\{x\in {\mathbb R^d} \mid |x|=R\}$ 
and $\mu=\beta\delta_R$ for $\beta>0$. 
Assume that the branching mechanism $\{p_n(x)\}_{n=0}^{\infty}$ is spherically symmetric 
and satisfies $p_0(x)+p_2(x)\equiv 1$. 
We use the notation $p_n(x)=p_n(|x|)$ and let $p=p_2(R)$.  
For $d=2$, Theorems \ref{thm:lim-dist}, \ref{thm:equiv} and \ref{thm:yaglom} are valid if $1/2<p\leq 1$. 
For $d\geq 3$, we know that $\lambda<0$ if and only if
\begin{equation*}
(2p-1)\beta R>\frac{d-2}{2}
\end{equation*}
(see, e.g., \cite[Example 2.14]{S18} and references therein). 
Theorems \ref{thm:lim-dist}, \ref{thm:equiv} and Corollary \ref{thm:yaglom} 
are valid under this condition. 
\end{exam}

\begin{exam}\rm 
Assume that the function $R(x)$ in \eqref{eq:exp-offspring} satisfies $R\not\equiv 0$. 
Let $V(x)$ be a nonnegative function on ${\mathbb R}^d$ such that 
$V\not\equiv 0$ and 
\begin{equation*}
(1\vee R(x))V(x)\leq {\bf 1}_{\{0<|x|\leq r_0\}}|x|^l
\end{equation*}
for some $r_0>0$ and $l\in {\mathbb R}$. 
Let $\mu({\rm d}x)=\beta V(x)\,{\rm d}x$ for $\beta>0$. 
If $d=1$ and $l>-1$, or if $d\geq 2$ and $l>-2$, 
then $\mu$ and $\nu_R$ are Kato class measures with compact support in ${\mathbb R}^d$ 
(see, e.g., \cite[Examples 2.2 and  2.15 (ii)]{S18}). 
Moreover,
there exists $\beta_*>0$ such that $\lambda<0$ if and only if $\beta>\beta_*$ 
(see, e.g., \cite[Example 2.15 (ii)]{S18} and references therein). 
Theorems \ref{thm:lim-dist}, \ref{thm:equiv} and Corollary \ref{thm:yaglom} 
are valid under this condition. 
\end{exam}

\section{Estimates of Feynman-Kac semigroups}\label{sect:fk-est}
Throughout this section, $\nu^+$ and $\nu^{-}$ are Kato class measures on ${\mathbb R}^d$ 
and $\nu=\nu^{+}-\nu^{-}$. 
When $\nu^+$ and $\nu^{-}$ belong to ${\cal K}_{\infty}(1)$, we let $\lambda=\lambda(\nu)$. 
If $\lambda<0$, then we also let $\lambda_2=\lambda_2(\nu)$. 

\subsection{Preliminary lemma}
In this subsection, we prove a lemma on the density function and principal eigenfunction 
associated with the Feynman-Kac semigroup. 
Denote by $S^{d-1}$ and ${\rm d}\theta$, respectively, 
the surface of a unit ball in ${\mathbb R}^d$ and the surface measure on $S^{d-1}$. 
In particular, if $d=1$, then ${\rm d}\theta$ is the Dirac measure on $S^0=\{-1,1\}$.  
Let $\langle \cdot , \cdot \rangle$ be the standard inner product in ${\mathbb R}^d$.

\begin{lem}\label{lem:int-eigen} 
Let $\nu$ be a signed measure on ${\mathbb R}^d$ 
such that $\nu=\nu^{+}-\nu^{-}$ 
for some $\nu^{+}, \nu^{-}\in {\cal K}$.
\begin{itemize}
\item[{\rm (i)}] For any $x,y\in {\mathbb R}^d$ and $t>0$,
\begin{equation*}
\begin{split}
p_t^{\nu}(x,y)-p_t(x,y)
&=\int_0^t \left(\int_{{\mathbb R}^d}p_s^{\nu}(x,z)p_{t-s}(z,y)\,\nu({\rm d}z)\right)\,{\rm d}s\\
&=\int_0^t \left(\int_{{\mathbb R}^d}p_s(x,z)p_{t-s}^{\nu}(z,y)\,\nu({\rm d}z)\right)\,{\rm d}s.
\end{split}
\end{equation*} 
\item[{\rm (ii)}] If $\nu^{+}, \nu^{-}\in {\cal K}_{\infty}(1)$ and $\lambda<0$, then  
there exists $C>0$ such that for any $x,y\in {\mathbb R}^d$ and $t\geq 1$,  
$$|p_t^{\nu}(x,y)-e^{-\lambda t}h(x)h(y)|\leq Ce^{-\lambda_2 t}.$$

\item[{\rm (iii)}] If $\nu^+, \nu^{-}\in {\cal K}_{\infty}(1)$ and $\lambda<0$, 
then for any $x\in {\mathbb R}^d$,
$$
h(x)=\int_{{\mathbb R}^d}G_{-\lambda}(x,y)h(y)\,\nu({\rm d}y).
$$

\item[{\rm (iv)}] 
If $\nu^{+}$ and $\nu^{-}$ are compactly supported in ${\mathbb R}^d$ 
and $\lambda<0$, then the constant 
\begin{equation}\label{eq:const}
c_d:=\frac{(\sqrt{-2\lambda})^{(d-5)/2}}{(2\pi)^{(d-1)/2}}
\int_{{\mathbb R}^d}\left(\int_{S^{d-1}}e^{\sqrt{-2\lambda}\langle \theta,z\rangle}\,{\rm d}\theta\right)
h(z)\,\nu({\rm d}z)
\end{equation}
is positive and 
\begin{equation}\label{eq:int-h-0}
\int_{|y|>R}h(y)\,{\rm d}y
\sim c_d e^{-\sqrt{-2\lambda}R}R^{(d-1)/2} \ (R\rightarrow\infty).
\end{equation}
\end{itemize}
\end{lem}

\begin{proof} 
(i) 
Let $A_t:=A_t^{\nu}$. Since 
$$e^{A_t}-1=e^{A_t}(1-e^{-A_t})=\int_0^t e^{A_t-A_s}\,{\rm d}A_s=\int_0^t e^{A_{t-s}\circ\theta_s}\,{\rm d}A_s,$$
the Markov property implies that  for any $f\in {\cal B}_b({\mathbb R}^d)$,
\begin{equation}\label{eq:af-dist}
\begin{split}
E_x\left[e^{A_t}f(B_t)\right]-E_x\left[f(B_t)\right]
=E_x\left[\int_0^t E_{B_s}\left[e^{A_{t-s}}f(B_{t-s})\right]\,{\rm d}A_s\right]
\end{split}
\end{equation}
(see \cite[p.186, Exercise 1.13]{RY99} and \cite[(3.8)]{S18}). 
Then for any $\alpha>\beta(\nu)$, 
we have by the Fubini theorem and \cite[p.229, (5.1.14)]{FOT11},
\begin{equation}\label{eq:laplace-0}
\begin{split}
&\int_0^{\infty}e^{-\alpha t}E_x\left[\int_0^t E_{B_s}\left[e^{A_{t-s}}f(B_{t-s})\right]\,{\rm d}A_s\right]\,{\rm d}t\\
&=E_x\left[\int_0^{\infty}e^{-\alpha s}
\left(\int_s^{\infty} e^{-\alpha (t-s)}E_{B_s}[e^{A_{t-s}}f(B_{t-s})]\,{\rm d}t\right)\,{\rm d}A_s\right]\\
&=E_x\left[\int_0^{\infty}e^{-\alpha s}G_{\alpha}^{\nu}f(B_s)\,{\rm d}A_s\right]
=\int_{{\mathbb R}^d}G_{\alpha}(x,z)G_{\alpha}^{\nu}f(z)\,\nu({\rm d}z).
\end{split}
\end{equation}
In the same way,  we also have
\begin{equation*}
\int_0^{\infty}e^{-\alpha t}
\left[\int_0^t\left(\int_{{\mathbb R}^d} p_s(x,z)p_{t-s}^{\nu}f(z)\,\nu({\rm d}z)\right)\,{\rm d}s\right]\,{\rm d}t
=\int_{{\mathbb R}^d}G_{\alpha}(x,z)G_{\alpha}^{\nu}f(z)\,\nu({\rm d}z).
\end{equation*}
Hence by \eqref{eq:af-dist} and \eqref{eq:laplace-0},
\begin{equation*}
\begin{split}
G_{\alpha}^{\nu}f(x)-G_{\alpha}f(x)
&=\int_0^{\infty}e^{-\alpha t}E_x\left[\int_0^t E_{B_s}\left[e^{A_{t-s}}f(B_{t-s})\right]\,{\rm d}A_s\right]\,{\rm d}t\\
&=\int_0^{\infty}e^{-\alpha t}
\left[\int_0^t\left(\int_{{\mathbb R}^d} p_s(x,z)p_{t-s}^{\nu}f(z)\,\nu({\rm d}z)\right)\,{\rm d}s\right]\,{\rm d}t,
\end{split}
\end{equation*}
which implies that 
\begin{equation}\label{eq:semi-dist}
p_t^{\nu}f(x)-p_tf(x)
=\int_0^t \left(\int_{{\mathbb R}^d}p_s(x,z)p_{t-s}^{\nu}f(z)\,\nu({\rm d}z)\right)\,{\rm d}s.
\end{equation}
We then get (i) by noting that $p_t(x,y)=p_t(y,x)$ and  $p_t^{\nu}(x,y)=p_t^{\nu}(y,x)$.

(ii) For  fixed $y\in {\mathbb R}^d$ and $s>0$, let 
\begin{equation*}
\varphi(z)=p_s^{\nu}(z,y)-e^{-\lambda s}h(z)h(y).
\end{equation*}
Since $h$ is bounded on ${\mathbb R}^d$, there exists $c_1(s)>0$ by \eqref{eq:heat-comp} such that 
for any $y\in {\mathbb R}^d$, $$\|\varphi\|_{L^2({\mathbb R}^d)}\leq c_1(s).$$
By noting that $\int_{{\mathbb R}^d}\varphi(z) h(z)\,{\rm d}z=0$, 
there exists $c_2(s)>0$ by \eqref{eq:poincare} such that 
for any $x,y\in {\mathbb R}^d$ and $t\geq 1/2$, 
\begin{equation*}
|p_t^{\nu}\varphi(x)|=|p_{t+s}^{\nu}(x,y)-e^{-\lambda(t+s)}h(x)h(y)|\leq c_2(s)e^{-\lambda_2 t}.
\end{equation*}
The proof is complete by taking $s=1/2$ and then by replacing $t+1/2$ with $t$.

(iii) 
We have by \eqref{eq:eigen-eq}, \eqref{eq:semi-dist} and the Fubini theorem, 
\begin{equation*}
\begin{split}
&e^{-\lambda t}h(x)-p_th(x)
=p_t^{\nu}h(x)-p_th(x)
=\int_0^t\left(\int_{{\mathbb R}^d} p_s(x,z)p_{t-s}^{\nu}h(z)\,\nu({\rm d}z)\right)\,{\rm d}s\\
&=\int_0^t\left(\int_{{\mathbb R}^d} p_s(x,z)e^{-\lambda(t-s)}h(z)\,\nu({\rm d}z)\right)\,{\rm d}s
=e^{-\lambda t}\int_{{\mathbb R}^d}\left(\int_0^t e^{\lambda s}p_s(x,z)\,{\rm d}s\right)h(z)\,\nu({\rm d}z).
\end{split}
\end{equation*}
Then (iii) follows by dividing both sides above by $e^{-\lambda t}$ 
and then by letting $t\rightarrow\infty$.

(iv) \ By (iii), 
\begin{equation}\label{eq:int-resolvent}
\begin{split}
&\int_{|y|>R}h(y)\,{\rm d}y
=\int_{|y|>R}\left(\int_{{\mathbb R}^d}G_{-\lambda}(y,z)h(z)\nu({\rm d}z)\right)\,{\rm d}y\\
&=\int_{|y|>R}\left(\int_{{\mathbb R}^d}G_{-\lambda}(y,z)h(z)\nu^{+}({\rm d}z)\right)\,{\rm d}y
-\int_{|y|>R}\left(\int_{{\mathbb R}^d}G_{-\lambda}(y,z)h(z)\nu^{-}({\rm d}z)\right)\,{\rm d}y.
\end{split}
\end{equation}
Since $\nu^{+}$ and $\nu^{-}$ are compactly supported in ${\mathbb R}^d$, 
we see by \eqref{eq:resolvent} that
\begin{equation}\label{eq:int-asymp-1}
\begin{split}
&\int_{|y|>R}\left(\int_{{\mathbb R}^d}G_{-\lambda}(y,z)h(z)\nu^{\pm}({\rm d}z)\right)\,{\rm d}y\\
&\sim \frac{(\sqrt{-2\lambda})^{(d-3)/2}}{(2\pi)^{(d-1)/2}}
\int_{|y|>R}\left(\int_{{\mathbb R}^d}
\frac{e^{-\sqrt{-2\lambda}|y-z|}}{|y-z|^{(d-1)/2}}h(z)\,\nu^{\pm}({\rm d}z)\right)\,{\rm d}y \ (R\rightarrow\infty).
\end{split}
\end{equation}

Let us calculate the integral in the last term of \eqref{eq:int-asymp-1}. 
Let $(r,\theta)\in (0,\infty)\times S^{d-1}$ be the polar decomposition in ${\mathbb R}^d$. 
Then 
$$\sup_{z\in \supp[|\nu|], \theta\in S^{d-1}}\left|\frac{r}{|(r,\theta)-z|}-1\right|\rightarrow 0 \ (r\rightarrow\infty).$$
Since 
$$\sup_{z\in \supp[|\nu|], \theta\in S^{d-1}}\left||(r,\theta)-z|-r+\langle \theta, z\rangle \right|
\rightarrow 0 \ (r\rightarrow\infty)$$
and there exists $c>0$ such that  $|e^x-1|\leq 2|x| \ (|x|\leq c)$, we also have for any $\alpha>0$,
$$\sup_{z\in \supp[|\nu|], \theta\in S^{d-1}}\left|e^{-\sqrt{2\alpha}(|(r,\theta)-z|-r)}
-e^{\sqrt{2\alpha}\langle \theta,z\rangle} \right|\rightarrow 0 \ (r\rightarrow\infty).$$
Therefore,  
$$\sup_{z\in \supp[|\nu|], \theta\in S^{d-1}}\left|e^{-\sqrt{2\alpha}(|(r,\theta)-z|-r)}\left(\frac{r}{|(r,\theta)-z|}\right)^{(d-1)/2}
-e^{\sqrt{2\alpha}\langle \theta,z\rangle} \right|\rightarrow 0 \ (r\rightarrow\infty),$$
which implies that 
\begin{equation}\label{eq:int-rad}
\begin{split}
&\int_{|y|>R}\left(\int_{{\mathbb R}^d}\frac{e^{-\sqrt{-2\lambda}|y-z|}}{|y-z|^{(d-1)/2}}
h(z)\,\nu^{\pm}({\rm d}z)\right)\,{\rm d}y\\
&=\int_R^{\infty} \left[\int_{S^{d-1}}
\left(\int_{{\mathbb R}^d}\frac{e^{-\sqrt{-2\lambda}|(r,\theta)-z|}}{|(r,\theta)-z|^{(d-1)/2}}h(z)\,\nu^{\pm}({\rm d}z)\right)
\,{\rm d}\theta\right]r^{d-1}\,{\rm d}r\\
&=\int_R^{\infty} \left[\int_{{\mathbb R}^d}\left(\int_{S^{d-1}}
\frac{e^{-\sqrt{-2\lambda}|(r,\theta)-z|}}{|(r,\theta)-z|^{(d-1)/2}}
\,{\rm d}\theta\right)h(z)\,\nu^{\pm}({\rm d}z)\right]r^{d-1}\,{\rm d}r\\
&\sim \int_R^{\infty} e^{-\sqrt{-2\lambda}r}r^{(d-1)/2}\,{\rm d}r 
\int_{{\mathbb R}^d}\left(\int_{S^{d-1}}
e^{\sqrt{-2\lambda}\langle \theta,z\rangle}\,{\rm d}\theta \right)
h(z)\,\nu^{\pm}({\rm d}z)
\ (R\rightarrow\infty).
\end{split}
\end{equation}

Noting that 
\begin{equation}\label{eq:asymp-r}
\int_R^{\infty} e^{-\sqrt{-2\lambda}r}r^{(d-1)/2}\,{\rm d}r 
\sim \frac{1}{\sqrt{-2\lambda}}e^{-\sqrt{-2\lambda}R}R^{(d-1)/2}
\ (R\rightarrow\infty),
\end{equation}
we obtain by \eqref{eq:int-asymp-1} and \eqref{eq:int-rad}, 
\begin{equation}\label{eq:int-asymp-2}
\int_{|y|>R}\left(\int_{{\mathbb R}^d}G_{-\lambda}(y,z)h(z)\nu^{\pm}({\rm d}z)\right)\,{\rm d}y
\sim  c_d^{\pm}e^{-\sqrt{-2\lambda}R}R^{(d-1)/2}\ (R\rightarrow\infty)
\end{equation}
for 
\begin{equation*}
c_d^{\pm}=\frac{(\sqrt{-2\lambda})^{(d-5)/2}}{(2\pi)^{(d-1)/2}}
\int_{{\mathbb R}^d}\left(\int_{S^{d-1}}
e^{\sqrt{-2\lambda}\langle \theta,z\rangle}\,{\rm d}\theta \right)
h(z)\,\nu^{\pm}({\rm d}z).
\end{equation*}
Here we note that by \eqref{eq:est-eigen} and \eqref{eq:asymp-r}, there exist  positive constants $c_1$, $c_2$, $R_0$ 
such that for any $R\geq R_0$,
\begin{equation*}
c_1 e^{-\sqrt{-2\lambda}R}R^{(d-1)/2}
\leq \int_{|y|>R}h(y)\,{\rm d}y\leq c_2 e^{-\sqrt{-2\lambda}R}R^{(d-1)/2}.
\end{equation*}
Hence the proof is complete by \eqref{eq:int-resolvent} and \eqref{eq:int-asymp-2}. 
\end{proof}

\begin{rem}\label{rem:one-dim}\rm 
Let $\nu$ satisfy the condition in Lemma \ref{lem:int-eigen} (iv). 
\begin{enumerate}
\item[(i)] 
For $R>0$ and $\Theta\subset S^{d-1}$, let 
$C_{\Theta}(R)=\left\{x\in {\mathbb R}^d \mid |x|>R, x/|x|\in \Theta\right\}$. 
In the same way as (iv), we get
\begin{equation*}
\int_{C_{\Theta}(R)} h(y)\,{\rm d}y
\sim c_{d,\Theta} e^{-\sqrt{-2\lambda}R} 
R^{(d-1)/2} \ (R\rightarrow\infty)
\end{equation*}
for
\begin{equation*}
c_{d,\Theta}=\frac{(\sqrt{-2\lambda})^{(d-5)/2}}{(2\pi)^{(d-1)/2}}
\int_{{\mathbb R}^d}\left(\int_{\Theta} e^{\sqrt{-2\lambda}\langle \theta,z\rangle}\,{\rm d}\theta\right)
h(z)\,\nu({\rm d}z). 
\end{equation*}
In particular, if $d=1$ and $\Theta=\{1\}$, then   
\begin{equation}\label{eq:int-h}
\int_R^{\infty}h(y)\,{\rm d}y
\sim c_0 e^{-\sqrt{-2\lambda}R} \ (R\rightarrow\infty)
\end{equation}
for 
\begin{equation}\label{eq:const-1-0}
c_0=\frac{1}{-2\lambda}
\int_{{\mathbb R}}e^{\sqrt{-2\lambda}z}h(z)\,\nu({\rm d}z).
\end{equation}
\item[(ii)] Let $\mu$ be a positive Radon measure on ${\mathbb R}^d$. 
For $\alpha>0$, let $G_{\alpha}\mu$ be 
the $\alpha$-potential of $\mu$ defined by  
\begin{equation*}
G_{\alpha}\mu(x)=\int_{{\mathbb R}^d}G_{\alpha}(x,y)\,\mu({\rm d}y).
\end{equation*}
Then by \eqref{eq:resolvent}, \eqref{eq:est-eigen} and Lemma \ref{lem:int-eigen} (iii), 
there exist positive constants $c_1$ and $c_2$ such that for any $x\in {\mathbb R}^d$, 
\begin{equation}\label{eq:potential}
c_1h(x)\leq G_{-\lambda}|\nu|(x)\leq c_2h(x). 
\end{equation}
\end{enumerate}
\end{rem}

\subsection{Pointwise estimates}
Let $\nu^{+}$ and $\nu^{-}$ be Kato class measures with compact support in ${\mathbb R}^d$ 
such that $\lambda<0$. 
Define 
\begin{equation*}
q_t(x,y)=p_t^{\nu}(x,y)-p_t(x,y)-e^{-\lambda t}h(x)h(y)
\end{equation*}
so that for $R>0$,
\begin{equation}\label{eq:q_t-0}
E_x\left[e^{A_t^{\nu}};|B_t|>R\right]
=P_x(|B_t|>R)+e^{-\lambda t}h(x)\int_{|y|>R}h(y)\,{\rm d}y
+\int_{|y|>R}q_t(x,y)\,{\rm d}y.
\end{equation}
In this subsection, we evaluate the last term in the right hand side above 
by Lemma \ref{lem:int-eigen}.

For $x\in {\mathbb R}^d$ and $R>0$, let $B_x(R):=\{y\in {\mathbb R}^d \mid |y-x|<R\}$. 
We fix $M>0$ so that $\supp[|\nu|]\subset B_0(M)$. 
For $c>0$ with $c\geq -\lambda_2$, we define 
\begin{equation}\label{eq:def-i}
I_c(t,R)=
\begin{cases}
e^{c t-\sqrt{2c}R}R^{(d-1)/2} & (\lambda_2<0),\\
t P_0(|B_t|> R-M)\wedge e^{ct-\sqrt{2c}R}R^{(d-1)/2} & (\lambda_2=0).
\end{cases}
\end{equation}
We also define 
\begin{equation}\label{eq:def-j}
J(t,R)=
e^{-\lambda t-\sqrt{-2\lambda}R}R^{(d-1)/2}
\int_{(\sqrt{-2\lambda}t-R)/\sqrt{2t}}^{\infty}e^{-v^2}\,{\rm d}v.
\end{equation}

\begin{prop}\label{prop:est-int}
Let $\nu^{+}$ and $\nu^{-}$ be Kato class measures with compact support in ${\mathbb R}^d$ 
such that $\lambda<0$.
\begin{enumerate}
\item For any $c>0$ with $c\geq -\lambda_2$, there exists $C_1>0$ such that 
for any $x\in {\mathbb R}^d$, $t\geq 1$ and $R>2M$,
\begin{equation}\label{e:int-upper-conclusion}
\left|\int_{|y|>R}q_t(x,y)\,{\rm d}y\right|
\leq C_1\left(h(x) P_0(|B_t|>R-M)+I_c(t,R)+h(x)J(t,R)\right).
\end{equation}
\item There exists $C_2>0$ such that 
for any $t\geq 1$,
\begin{equation}\label{eq:fk-asymp}
\sup_{x\in {\mathbb R}^d}\left|\int_{{\mathbb R}^d}q_t(x,y)\,{\rm d}y\right|
\leq C_2(t\vee e^{-\lambda_2 t}).
\end{equation}
\end{enumerate}
\end{prop}

\begin{rem}\label{rem:uniform}\rm 
Let $f\in {\cal B}_b({\mathbb R}^d)$. Since 
\begin{equation*}
E_x\left[e^{A_t^{\nu}}f(B_t)\right]
=E_x\left[f(B_t)\right]+e^{-\lambda t}h(x)\int_{{\mathbb R}^d}f(y)h(y)\,{\rm d}y+\int_{{\mathbb R}^d}q_t(x,y)f(y)\,{\rm d}y,
\end{equation*}
there exists $C>0$ by Proposition \ref{prop:est-int} (ii) such that 
for any $t\geq 1$, 
\begin{equation*}
\sup_{x\in {\mathbb R}^d}
\left|e^{\lambda t}E_x\left[e^{A_t^{\nu}}f(B_t)\right]-h(x)\int_{{\mathbb R}^d}f(y)h(y)\,{\rm d}y\right|
\leq C\|f\|_{\infty}e^{\lambda t}(t\vee e^{-\lambda_2 t}).
\end{equation*}
As $\lambda<\lambda_2\leq 0$, the right hand side above goes to $0$ as $t\rightarrow\infty$.
\end{rem}

To show Proposition \ref{prop:est-int}, we deform $q_t(x,y)$ as follows. 
For $t\geq 1$, we have  by Lemma \ref{lem:int-eigen} (i),
\begin{equation*}
\begin{split}
&p_t^{\nu}(x,y)-p_t(x,y)\\
&=\int_0^1 \left(\int_{{\mathbb R}^d}p_s^{\nu}(x,z)p_{t-s}(z,y)\,\nu({\rm d}z)\right)\,{\rm d}s
+\int_1^t \left(\int_{{\mathbb R}^d}p_s^{\nu}(x,z)p_{t-s}(z,y)\,\nu({\rm d}z)\right)\,{\rm d}s.
\end{split}
\end{equation*}
The second term in the right hand side above is equal to 
\begin{equation*}
\begin{split}
&\int_1^t \left[\int_{{\mathbb R}^d}
\left(p_s^{\nu}(x,z)-e^{-\lambda s}h(x)h(z)\right)p_{t-s}(z,y)\,\nu({\rm d}z)\right]\,{\rm d}s\\
&+h(x)\int_1^t e^{-\lambda s}\left(\int_{{\mathbb R}^d}h(z)p_{t-s}(z,y)\,\nu({\rm d}z)\right)\,{\rm d}s.
\end{split}
\end{equation*}
Since the Fubini theorem and Lemma \ref{lem:int-eigen} (iii) yield that 
\begin{equation*}
\int_0^{\infty} e^{\lambda s}\left(\int_{{\mathbb R}^d}h(z)p_s(z,y)\,\nu({\rm d}z)\right)\,{\rm d}s
=\int_{{\mathbb R}^d}G_{-\lambda}(y,z)h(z)\,\nu({\rm d}z)=h(y),
\end{equation*}
we obtain by the change of variables,
\begin{equation*}
\begin{split}
&\int_1^t e^{-\lambda s}\left(\int_{{\mathbb R}^d}h(z)p_{t-s}(z,y)\,\nu({\rm d}z)\right)\,{\rm d}s
=e^{-\lambda t}\int_0^{t-1} e^{\lambda s}
\left(\int_{{\mathbb R}^d}h(z)p_s(z,y)\,\nu({\rm d}z)\right)\,{\rm d}s\\
&=e^{-\lambda t}\int_0^{\infty} e^{\lambda s}\left(\int_{{\mathbb R}^d}h(z)p_s(z,y)\,\nu({\rm d}z)\right)\,{\rm d}s
-e^{-\lambda t}\int_{t-1}^{\infty} e^{\lambda s}\left(\int_{{\mathbb R}^d}h(z)p_s(z,y)\,\nu({\rm d}z)\right)\,{\rm d}s\\
&=e^{-\lambda t}h(y)
-e^{-\lambda t}\int_{t-1}^{\infty} e^{\lambda s}
\left(\int_{{\mathbb R}^d}h(z)p_s(z,y)\,\nu({\rm d}z)\right)\,{\rm d}s
\end{split}
\end{equation*}
and thus 
\begin{equation*}
\begin{split}
q_t(x,y)&=p_t^{\nu}(x,y)-p_t(x,y)-e^{-\lambda t}h(x)h(y)\\
&=\int_0^1 \left(\int_{{\mathbb R}^d}p_s^{\nu}(x,z)p_{t-s}(z,y)\,\nu({\rm d}z)\right)\,{\rm d}s\\
&+\int_1^t \left[\int_{{\mathbb R}^d}\left(p_s^{\nu}(x,z)-e^{-\lambda s}h(x)h(z)\right)p_{t-s}(z,y)\,\nu({\rm d}z)\right]\,{\rm d}s\\
&-e^{-\lambda t}h(x)\int_{t-1}^{\infty} e^{\lambda s}\left(\int_{{\mathbb R}^d}h(z)p_s(z,y)\,\nu({\rm d}z)\right)\,{\rm d}s.
\end{split}
\end{equation*}
Then by the Fubini theorem,
\begin{equation}\label{eq:q_t}
\begin{split}
\int_{{|y|>R}}q_t(x,y)\,{\rm d}y
&=\int_0^1 \left(\int_{{\mathbb R}^d}p_s^{\nu}(x,z)P_z(|B_{t-s}|>R)\,\nu({\rm d}z)\right)\,{\rm d}s\\
&+\int_1^t  \left[\int_{{\mathbb R}^d}\left(p_s^{\nu}(x,z)-e^{-\lambda s}h(x)h(z)\right)P_z(|B_{t-s}|>R)\,\nu({\rm d}z)\right]\,{\rm d}s\\
&-e^{-\lambda t}h(x)
\int_{t-1}^{\infty} e^{\lambda s}\left(\int_{{\mathbb R}^d}h(z)P_z(|B_s|>R)\,\nu({\rm d}z)\right)\,{\rm d}s\\
&={\rm (I)}+{\rm (II)}-{\rm (III)}.
\end{split}
\end{equation}

We first discuss the upper bound of ${\rm (I)}$. 
\begin{lem}\label{lem:i}
Under the same setting as in Proposition {\rm \ref{prop:est-int}}, 
there exists $C>0$ such that 
for any $x\in {\mathbb R}^d$, $t\geq 1$ and $R>M$, 
\begin{equation}\label{eq:int-upper-1-2}
|{\rm (I)}|\leq C h(x) P_0(|B_t|>R-M).
\end{equation}
\end{lem}
\begin{proof}
For any $R>M$ and $z\in \supp[|\nu|]$, 
\begin{equation}\label{eq:comp-upper}
P_z(|B_{t-s}|>R)\leq P_0(|B_{t-s}|>R-M)\leq P_0(|B_t|>R-M)
\end{equation}
and thus
\begin{equation}\label{eq:int-upper-1}
|{\rm (I)}|\leq P_0(|B_t|>R-M)
\int_0^1 \left(\int_{{\mathbb R}^d}p_s^{\nu}(x,z)\,|\nu|({\rm d}z)\right)\,{\rm d}s.
\end{equation}
Then by \eqref{eq:heat-comp} and \eqref{eq:potential}, 
we have for any $x\in {\mathbb R}^d$,
\begin{equation}\label{eq:upper-01}
\begin{split}
&\int_0^1 \left(\int_{{\mathbb R}^d}p_s^{\nu}(x,z)\,|\nu|({\rm d}z)\right)\,{\rm d}s
\leq c_1 \int_0^1 \left(\int_{{\mathbb R}^d}p_{c_2s}(x,z)\,|\nu|({\rm d}z)\right)\,{\rm d}s\\
&\leq c_3 \int_0^{\infty} \left(\int_{{\mathbb R}^d}e^{\lambda s}p_s(x,z)\,|\nu|({\rm d}z)\right)\,{\rm d}s
=c_3G_{-\lambda}|\nu|(x)\leq c_4 h(x).
\end{split}
\end{equation}
Substituting this into \eqref{eq:int-upper-1}, we obtain \eqref{eq:int-upper-1-2}.
\end{proof}

We next discuss the bound of ${\rm (II)}$. 
\begin{lem}\label{lem:ii} 
Under the same setting as in Proposition {\rm \ref{prop:est-int}}, 
the following three assertions hold. 
\begin{itemize}
\item[{\rm (i)}] 
For any $c>0$ with $c\geq -\lambda_2$, there exists $C_1>0$ 
such that for any $t\geq 1$ and $R>2M$, 
\begin{equation*}
\int_1^t \left(\int_{{\mathbb R}^d}e^{-\lambda_2 s}P_z(|B_{t-s}|>R)\,|\nu|({\rm d}z)\right)\,{\rm d}s
\leq  C_1 e^{c t-\sqrt{2c}R}R^{(d-1)/2}.
\end{equation*}
\item[{\rm (ii)}] 
If $\lambda_2=0$, then there exists $C_2>0$
such that for any $t\geq 1$ and $R>M$, 
\begin{equation*}
\int_1^t \left(\int_{{\mathbb R}^d}P_z(|B_{t-s}|>R)\,|\nu|({\rm d}z)\right)\,{\rm d}s
\leq  
C_2 tP_0(|B_t|>R-M).
\end{equation*}

\item[{\rm (iii)}] 
For any $c>0$ with $c\geq -\lambda_2$, there exists $C_3>0$ 
such that for any $x\in {\mathbb R}^d$, $t\geq 1$ and $R>M$, 
\begin{equation}\label{eq:int-upper-2}
|{\rm (II)}|\leq C_3I_c(t,R).
\end{equation}
\end{itemize}
\end{lem}

\begin{proof}
By Lemma \ref{lem:int-eigen} (ii), there exists $c_1>0$ such that 
for any $x\in {\mathbb R}^d$ and $t\geq 1$,
\begin{equation*}
\begin{split}
|{\rm (II)}|&\leq \int_1^t \left(\int_{{\mathbb R}^d}\left|p_s^{\nu}(x,z)-e^{-\lambda s}h(x)h(z)\right|
P_z(|B_{t-s}|>R)\,|\nu|({\rm d}z)\right)\,{\rm d}s\\
&\leq c_1\int_1^t \left(\int_{{\mathbb R}^d}e^{-\lambda_2 s}P_z(|B_{t-s}|>R)\,|\nu|({\rm d}z)\right)\,{\rm d}s. 
\end{split}
\end{equation*}
Therefore, (iii) follows from (i) and (ii).

Let us show (i) and (ii). 
By the change of variables and \eqref{eq:comp-upper},
\begin{equation}\label{eq:int-spectral}
\begin{split}
&\int_1^t \left(\int_{{\mathbb R}^d}e^{-\lambda_2 s}P_z(|B_{t-s}|>R)\,|\nu|({\rm d}z)\right)\,{\rm d}s\\
&=
\int_0^{t-1} \left(\int_{{\mathbb R}^d}e^{-\lambda_2(t-s)}P_z(|B_s|>R)\,|\nu|({\rm d}z)\right)\,{\rm d}s\\
&\leq |\nu|({\mathbb R}^d)
\int_0^{t-1}e^{-\lambda_2(t-s)}P_0(|B_s|>R-M)\,{\rm d}s.
\end{split}
\end{equation}
Then for any $c>0$ with $c\geq -\lambda_2$, we have by \eqref{eq:resolvent},
\begin{equation*}
\begin{split}
&\int_0^{t-1}e^{-\lambda_2(t-s)}P_0(|B_s|>R-M)\,{\rm d}s
\leq \int_0^{\infty}e^{c(t-s)}P_0(|B_s|>R-M)\,{\rm d}s\\
&=e^{ct}\int_{|y|>R-M}G_c(0,y)\,{\rm d}y
\leq c_2 e^{ct}\int_{|y|>R-M} \frac{e^{-\sqrt{2c}|y|}}{|y|^{(d-1)/2}}\,{\rm d}y
\leq c_3 e^{ct-\sqrt{2c}R}R^{(d-1)/2},
\end{split}
\end{equation*}
which implies (i). 
If $\lambda_2=0$, then \eqref{eq:comp-upper} yields that 
$$
\int_0^{t-1}P_0(|B_s|>R-M)\,{\rm d}s
\leq tP_0(|B_t|>R-M).
$$
Hence (ii) follows by \eqref{eq:int-spectral}.
\end{proof}

We finally give an upper bound of ${\rm (III)}$. 
\begin{lem}\label{lem:iii}
Under the same setting as in Proposition {\rm \ref{prop:est-int}}, 
there exists $C>0$ for any $c\geq -\lambda_2$ 
such that for any $x\in {\mathbb R}^d$, $t\geq 1$ and $R>2M$, 
\begin{equation}\label{eq:int-upper-3}
|{\rm (III)}|\leq Ch(x)\left(P_0(|B_t|>R-M)+J(t,R)\right).
\end{equation}
\end{lem}

\begin{proof}
By \eqref{eq:comp-upper}, 
\begin{equation*}
\begin{split}
&\int_{t-1}^{\infty}e^{\lambda s}\left(\int_{{\mathbb R}^d} h(z)P_z(|B_s|>R)\,|\nu|({\rm d}z)\right)\,{\rm d}s\\
&\leq \int_{t-1}^{\infty} e^{\lambda s}P_0(|B_s|>R-M)\,{\rm d}s
\int_{{\mathbb R}^d}h(z)\,|\nu|({\rm d}z)
\leq c_1\int_t^{\infty} e^{\lambda s}P_0(|B_s|>R-M)\,{\rm d}s.
\end{split}
\end{equation*}
Then by the integration by parts formula, 
\begin{equation*}
\begin{split}
&\int_t^{\infty} e^{\lambda s}P_0(|B_s|>R-M)\,{\rm d}s\\
&=\frac{1}{-\lambda}e^{\lambda t} P_0(|B_t|>R-M)
+\frac{1}{-\lambda}\int_t^{\infty}e^{\lambda s}\frac{\partial }{\partial s}P_0(|B_s|>R-M)\,{\rm d}s.
\end{split}
\end{equation*}
Hence for any $x\in {\mathbb R}^d$, $t\geq 1$ and $R>M$,
\begin{equation}\label{eq:upper-iii}
\begin{split}
|{\rm (III)}|
\leq c_2 h(x)\left(P_0(|B_t|>R-M)
+e^{-\lambda t}\int_t^{\infty}e^{\lambda s}\frac{\partial }{\partial s}P_0(|B_s|>R-M)\,{\rm d}s\right).
\end{split}
\end{equation}

By the change of variables,
\begin{equation}
\begin{split}
&\int_t^{\infty}e^{\lambda s}\frac{\partial }{\partial s}P_0(|B_s|>R-M)\,{\rm d}s
=c_3(R-M)^d\int_t^{\infty}e^{\lambda s}e^{-(R-M)^2/(2s)}\frac{1}{s^{(d+2)/2}}\,{\rm d}s\\
&=c_4(R-M)^d \int_{\sqrt{t}}^{\infty}e^{\lambda u^2-(R-M)^2/(2u^2)}\frac{1}{u^{d+1}}\,{\rm d}u\\
&=c_4 e^{-\sqrt{-2\lambda}(R-M)}(R-M)^d\int_{\sqrt{t}}^{\infty}e^{-(\sqrt{-\lambda}u-(R-M)/(\sqrt{2}u))^2}\frac{1}{u^{d+1}}\,{\rm d}u.
\end{split}
\label{eq:deri-upper}
\end{equation}
We then see by \eqref{eq:negative-upper} and \eqref{eq:positive-upper} below that 
for any $t\geq 1$ and $R>2M$,
$$
e^{-\lambda t}\int_t^{\infty}e^{\lambda s}\frac{\partial }{\partial s}P_0(|B_s|>R-M)\,{\rm d}s
\leq c_5J(t,R).
$$
Hence the proof is complete by \eqref{eq:upper-iii}.
\end{proof}

\begin{proof}[Proof of Proposition {\rm \ref{prop:est-int}}] 
(i) is a consequence of Lemmas \ref{lem:i}, \ref{lem:ii} and \ref{lem:iii}.
We now show (ii).  If we take $R=0$ in \eqref{eq:q_t}, then   
\begin{equation*}
\begin{split}
\left| \int _{ {\mathbb R}^d } q_t (x,y) {\rm d} y \right|
&\le \int _0 ^1 \left( \int_{ {\mathbb R}^d } p_s ^{\nu} (x,z) |\nu| ({\rm d}z) \right) {\rm d}s \\
& + \int _1 ^t \left( \int _{ {\mathbb R}^d } \left| p_s ^{\nu} (x,z) - e^{-\lambda s} h(x) h(z) \right|
|\nu| ({\rm d}z) \right) {\rm d} s + c_1h(x)\int_{{\mathbb R}^d}h(z)\,|\nu|({\rm d}z).
\end{split}
\end{equation*}
By Lemma \ref{lem:int-eigen} (ii),
\[
\begin{split}
\int_1 ^t \left( \int _{ {\mathbb R}^d } \left| p_s ^{\nu} (x,z) - e^{-\lambda s} h(x) h(z) \right|
	|\nu| ({\rm d}z) \right) {\rm d} s 
& \le  c_2|\nu|({\mathbb R}^d)\int _1 ^t e^{-\lambda_2 s}\,{\rm d} s 
\le 
c_3 \left( t \vee e^{-\lambda _2 t} \right) .
\end{split}
\]
Combining this with \eqref{eq:upper-01}, we arrive at the conclusion.
\end{proof}

\subsection{Uniform estimates}
In this subsection, we show two lemmas related to the Feynman-Kac semigroups 
provided that the Brownian particle sits outside a ball with time dependent radius. 
Let $\nu^{+}$ and $\nu^{-}$ be Kato class measures with compact support in ${\mathbb R}^d$ 
such that $\lambda<0$. 
Let $a(t)$ be a function on $(0,\infty)$ such that 
$a(t)=o(t)$ as $t\rightarrow\infty$. 
For a fixed  $\delta\in (0, \sqrt{-2\lambda})$, 
we define $R(t)=\delta t+a(t)$ and 
$$\eta(t)=e^{-\lambda t}\int_{|y|>R(t)}h(y)\,{\rm d}y.$$ 
Then by Lemma \ref{lem:int-eigen} (iv),
\begin{equation}\label{eq:asymp-eta}
\eta(t)
\sim c_dR(t)^{(d-1)/2}e^{-\lambda t-\sqrt{-2\lambda}R(t)}
\sim c_d\delta^{(d-1)/2}t^{(d-1)/2}
e^{(-\lambda-\sqrt{-2\lambda}\delta)t-\sqrt{-2\lambda}a(t)} \ (t\rightarrow\infty).
\end{equation}

\begin{lem}\label{lem:exp-bound}
Let $\nu^{+}$ and $\nu^{-}$ be Kato class measures with compact support in ${\mathbb R}^d$ 
such that $\lambda<0$. 
Let $K$ be a compact set in ${\mathbb R}^d$ and $\delta\in (0,\sqrt{-2\lambda})$. 
Then for any $\alpha\in (0, 1-\delta/\sqrt{-2\lambda})$,  
there exist positive constants $C$ and $T$ such that 
for any $x\in K$, $t\geq T$ and $s\in[0,\alpha t]$,
\begin{equation*}
\begin{split}
\left|E_x\left[e^{A_{t-s}^{\nu}};|B_{t-s}|>R(t)\right]-e^{\lambda s}h(x)\eta(t)\right|
&\leq e^{-C t}e^{\lambda s}e^{-\lambda t-\sqrt{-2\lambda}R(t)}.
\end{split}
\end{equation*}
\end{lem}

\begin{proof}
Let $\omega_d$ be the area of the unit ball in ${\mathbb R}^d$ 
and let $M>0$ satisfy $K\cup \supp[|\nu|]\subset B_0(M)$. 
Then for any $x\in K$, large $t$ and $s\in [0,\alpha t]$, 
\begin{equation}\label{eq:upper-bm-0}
\begin{split}
P_x(|B_{t-s}|>R(t))\leq P_0(|B_{t-s}|>R(t)-M)
&=\frac{\omega_d}{(2\pi)^{d/2}}\int_{(R(t)-M)/\sqrt{t-s}}^{\infty}e^{-u^2/2}u^{d-1}\,{\rm d}u\\
&\leq c_1 \exp\left( -\frac{(R(t)-M)^2}{2(t-s)}\right)t^{(d-2)/2}
\end{split}
\end{equation}
because 
$$\frac{R(t)-M}{\sqrt{t-s}}
\geq \frac{R(t)-M}{\sqrt{t}}\rightarrow\infty \ (t\rightarrow\infty)$$
and 
\begin{equation}\label{eq:asymp-r-2}
\int_R^{\infty}e^{-u^2/2}u^{d-1}\,{\rm d}u\sim e^{-R^2/2}R^{d-2} \ (R\rightarrow\infty).
\end{equation}
Since the function 
\begin{equation*}
f(s)=-\lambda s-\frac{(R(t)-M)^2}{2(t-s)} \ (s\in [0,\alpha t])
\end{equation*} is increasing 
for all sufficiently large $t$, we have $f(s)\leq f(\alpha t)$ and thus
\begin{equation}\label{eq:ineq-exp-1}
\exp\left(-\lambda s-\frac{(R(t)-M)^2}{2(t-s)}\right)
\leq \exp\left(-\lambda \alpha t-\frac{(R(t)-M)^2}{2(1-\alpha)t}\right).
\end{equation}
We can further take $c_2>0$ so that 
\begin{equation}\label{eq:ineq-exp-2}
\exp\left(-\lambda \alpha t-\frac{(R(t)-M)^2}{2(1-\alpha)t}\right)
\leq e^{-c_2t}e^{-\lambda t-\sqrt{-2\lambda}R(t)}
\end{equation}
because 
\begin{equation*}
\begin{split}
&\left[- \lambda t - \sqrt{-2 \lambda} R(t)
-\left\{-\lambda \alpha t-\frac{(R(t)-M)^2}{2(1-\alpha)t}\right\}\right]/t\\
&=\left( \sqrt{-\lambda(1-\alpha)} - \frac{R(t) -M}{\sqrt{2(1-\alpha) }t} \right)^2 - 
\frac{\sqrt{-2\lambda} M}{t}
\to \frac{-\lambda}{1-\alpha}\left(1-\frac{\delta}{\sqrt{-2\lambda}}-\alpha\right)^2 \ (t \to \infty)
\end{split}
\end{equation*}
and the limit is positive. 
Hence by \eqref{eq:upper-bm-0}, \eqref{eq:ineq-exp-1} and \eqref{eq:ineq-exp-2}, 
we get 
\begin{equation}\label{eq:upper-bm}
\begin{split}
P_x(|B_{t-s}|>R(t))\leq e^{-c_3 t}e^{\lambda s}e^{-\lambda t-\sqrt{-2\lambda}R(t)}.
\end{split}
\end{equation}

Fix $c>0$ so that 
$$\left\{(-\lambda_2)\vee \left(\frac{\sqrt{2}\delta}{1-\alpha}-\sqrt{-\lambda}\right)^2\right\}
<c<-\lambda.$$
Then for any large $t$ and $s\in [0,\alpha t]$, 
we obtain by \eqref{eq:def-i} and direct calculation,
\begin{equation}\label{eq:upper-i}
I_c(t-s, R(t))\leq e^{c(t-s)-\sqrt{2c}R(t)}R(t)^{(d-1)/2}
\leq c_4e^{-c_5 t}e^{\lambda s}e^{-\lambda t-\sqrt{-2\lambda}R(t)}.
\end{equation}

For $t\geq 1$ and  $s\in [0,\alpha t]$, 
since $\alpha\in (0,1-\delta/\sqrt{-2\lambda})$, we get 
$$\frac{\sqrt{-2\lambda} (t-s)-R(t)}{\sqrt{2(t-s)}}
\geq \frac{\{\sqrt{-2\lambda}(1-\alpha)-\delta\}t-a(t)}{\sqrt{2t}}
\geq c_1\sqrt{t}.$$
Then by the relation 
$$\int_R^{\infty}e^{-v^2}\,{\rm d}v \sim \frac{e^{-R^2}}{2R} \ (R\rightarrow\infty),$$ 
we have for any large $t$ and $s\in [0,\alpha t]$,
$$
\int_{(\sqrt{-2\lambda} (t-s)-R(t))/\sqrt{2(t-s)}}^{\infty} e^{-v^2}\,{\rm d}v
\leq \int_{c_1\sqrt{t}}^{\infty} e^{-v^2}\,{\rm d}v\leq c_2e^{-c_3 t}/\sqrt{t},
$$
and thus by \eqref{eq:def-j},
\begin{equation}\label{eq:upper-j}
\begin{split}
J(t-s,R(t))
&=
e^{-\lambda(t-s)}e^{-\sqrt{-2\lambda}R(t)}R(t)^{(d-1)/2}
\int_{(\sqrt{-2\lambda} (t-s)-R(t))/\sqrt{2(t-s)}}^{\infty} e^{-v^2}\,{\rm d}v\\
&\leq e^{-c_4 t}e^{\lambda s}e^{-\lambda t-\sqrt{-2\lambda}R(t)}.
\end{split}
\end{equation}

By Proposition \ref{prop:est-int} with \eqref{eq:upper-bm}, \eqref{eq:upper-i} and \eqref{eq:upper-j}, 
there exists $T>0$ such that for any $x\in K$, $t\geq T$ and $s\in [0,\alpha t]$, 
\begin{equation*}
\begin{split}
&\left|\int_{|y|>R(t)}q_{t-s}(x,y)\,{\rm d}y\right|\\
&\leq c_5\left(h(x) P_0(|B_{t-s}|>R(t)-M)+I_c(t-s,R(t))+h(x)J(t-s,R(t))\right)\\
&\leq c_6e^{-c_7t}e^{\lambda s}e^{-\lambda t-\sqrt{-2\lambda}R(t)}.
\end{split}
\end{equation*}
We therefore obtain by \eqref{eq:q_t-0} and \eqref{eq:upper-bm},
\begin{equation}\label{eq:int-upper-conclusion-0}
\begin{split}
\left|E_x\left[e^{A_{t-s}^{\nu}};|B_{t-s}|>R(t)\right]-e^{\lambda s}h(x)\eta(t)\right|
&\leq P_x(|B_{t-s}|> R(t))+\left|\int_{|y|>R(t)}q_{t-s}(x,y)\,{\rm d}y\right|\\
&\leq c_8e^{-c_9 t}e^{\lambda s}e^{-\lambda t-\sqrt{-2\lambda}R(t)},
\end{split}
\end{equation}
which completes the proof. 
\end{proof}

Let $K$ be a compact set in ${\mathbb R}^d$. 
Then for any $\delta\in (0, \sqrt{-2\lambda})$ and 
$\alpha\in (0,1-\delta/\sqrt{-2\lambda})$, 
Lemma \ref{lem:exp-bound} and \eqref{eq:asymp-eta} imply 
that for any $x\in K$, large $t$ and $s\in [0,\alpha t]$,
\begin{equation}\label{eq:both}
E_x\left[e^{A_{t-s}^{\nu}};|B_{t-s}|>R(t)\right]
=e^{\lambda s}h(x)\eta(t)(1+\theta_{s,x}(t)),
\end{equation}
where $\theta_{s,x}(t)$ is a function on $(0,\infty)$ such that 
$|\theta_{s,x}(t)|\leq c_1e^{-c_2 t}$. 
In particular, if we take $s=0$ in \eqref{eq:both}, then 
\begin{equation}\label{eq:both-1}
E_x\left[e^{A_t^{\nu}};|B_t|>R(t)\right]
=h(x)\eta(t)(1+\theta_{0,x}(t)).
\end{equation}
This equality refines \cite[Proposition 3.1]{S18}  
and justifies an observation in \cite[(3.3)]{S18}. 
Note that $\nu$ in \eqref{eq:both-1} is allowed to be signed.  

\begin{lem}\label{lem:second-moment}
Let $\nu^{+}$ and $\nu^{-}$ be as in Lemma {\rm \ref{lem:exp-bound}}.
Let $K$ be a compact set in ${\mathbb R}^d$ 
and $\mu$ a Kato class measure with compact support in ${\mathbb R}^d$. 
Then there exist $C>0$ and $T\geq 1$ 
such that for any $x\in K$, $t\geq T$ and $s\in [0,t-1]$,
\begin{equation*}
E_x\left[\int_0^{t-s}e^{A_u^{\nu}}
E_{B_u}\left[e^{A_{t-s-u}^{\nu}};|B_{t-s-u}|>R(t)\right]^2\,{\rm d}A_u^{\mu}\right]
\leq C\left(e^{\lambda s}\eta(t)\right)^2.
\end{equation*}
\end{lem}

\begin{proof}
Let
\begin{equation*}
\begin{split}
&E_x\left[\int_0^{t-s}e^{A_u^{\nu}}E_{B_u}\left[e^{A_{t-s-u}^{\nu}};|B_{t-s-u}|>R(t)\right]^2
\,{\rm d}A_u^{\mu}\right]\\
&=E_x\left[\int_0^{t-s-1}e^{A_u^{\nu}}E_{B_u}\left[e^{A_{t-s-u}^{\nu}};|B_{t-s-u}|>R(t)\right]^2\,{\rm d}A_u^{\mu}\right]\\
&+E_x\left[\int_{t-s-1}^{t-s}e^{A_u^{\nu}}E_{B_u}\left[e^{A_{t-s-u}^{\nu}};|B_{t-s-u}|>R(t)\right]^2\,{\rm d}A_u^{\mu}\right]
={\rm (IV)}+{\rm (V)}.
\end{split}
\end{equation*}
We first discuss the upper bound of ${\rm (IV)}$. 
Let $M>0$ satisfy $K\cup\supp[|\nu|]\cup\supp[\mu]\subset B_0(M)$. 
Then for any large $t\geq 1$, $s\in [0,t-1]$ and $x\in B_0(M)$, 
since a calculation similar to \eqref{eq:upper-bm} implies that 
\begin{equation}\label{eq:int-upper-s}
P_x(|B_{t-s}|>R(t))\leq P_0(|B_{t-s}|>R(t)-M)
\leq c_1e^{-\lambda (t-s)}e^{-\sqrt{-2\lambda}R(t)}t^{(d-2)/2},
\end{equation}
it follows by \eqref{eq:def-i}, \eqref{eq:def-j}, 
Proposition \ref{prop:est-int} and \eqref{eq:asymp-eta} that 
\begin{equation*}
\begin{split}
&\left|\int_{|y|>R(t)}q_{t-s}(x,y)\,{\rm d}y\right|\\
&\leq c_2\left(h(x) P_0(|B_{t-s}|>R(t)-M)+I_{-\lambda}(t-s,R(t))+h(x)J(t-s,R(t))\right)\\
&\leq c_3e^{-\lambda (t-s)}e^{-\sqrt{-2\lambda}R(t)}t^{(d-1)/2}
\leq c_4e^{\lambda s}\eta(t).
\end{split}
\end{equation*}
Combining this with  \eqref{eq:q_t-0} and  \eqref{eq:int-upper-s},
we have for any $x\in \supp[\mu]$, large $t\geq 1$ and $s\in [0,t-1]$,
\begin{equation*}
\begin{split}
E_x\left[e^{A_{t-s}^{\nu}};|B_{t-s}|>R(t)\right]
&\leq P_x(|B_{t-s}|>R(t))
+e^{\lambda s}h(x)\eta(t)
+\left|\int_{|y|>R(t)}q_{t-s}(x,y)\,{\rm d}y\right|\\
&\leq c_5e^{\lambda s}\eta(t).
\end{split}
\end{equation*}
By the same argument as in \cite[Proposition 3.3 (i)]{CS07}, 
we also have
\begin{equation}\label{eq:gaugeable}
\sup_{x\in {\mathbb R}^d}E_x\left[\int_0^{\infty}e^{2\lambda u+A_u^{\nu}}\,{\rm d}A_u^{\mu}\right]<\infty.
\end{equation}
Therefore, for any $x\in K$, large $t\geq 1$ and $s\in[0,t-1]$, 
\begin{equation}\label{eq:second-moment-1}
\begin{split}
{\rm (IV)}
\leq c_6 E_x\left[\int_0^{t-s-1}e^{A_u^{\nu}}
\left(e^{\lambda(s+u)}\eta(t)\right)^2\,{\rm d}A_u^{\mu}\right]
&=c_6 \left(e^{\lambda s}\eta(t)\right)^2
E_x\left[\int_0^{t-s-1}e^{2\lambda u+A_u^{\nu}}\,{\rm d}A_u^{\mu}\right]\\
&\leq c_7\left(e^{\lambda s}\eta(t)\right)^2.
\end{split}
\end{equation}

We next discuss the upper bound of ${\rm (V)}$. 
By Lemma \ref{lem:int-eigen} (i), 
\eqref{eq:comp-upper} and \eqref{eq:upper-01}, 
we have for any $x\in B_0(M)$, $t\in [0,1]$ and $R>M$,
\begin{equation*}
\begin{split}
E_x\left[e^{A_t^{\nu}};|B_t|>R\right]
&=P_x(|B_t|>R)+\int_0^t \left[\int_{{\mathbb R}^d}p_u^{\nu}(x,y)P_y(|B_{t-u}|>R)\,\nu({\rm d}y)\right]\,{\rm d}u\\
&\leq P_0(|B_t|>R-M)\left[1+c_1\int_0^1\left(\int_{{\mathbb R}^d}p_u^{\nu}(x,y)\,|\nu|({\rm d}y)\right)\,{\rm d}u\right]\\
&\leq c_2 P_0(|B_1|>R-M).
\end{split}
\end{equation*}
Then by \eqref{eq:asymp-r-2}, we see that for any $x\in \supp[\mu]$, large $t\geq 1$ and $u\in [t-1,t]$, 
\begin{equation*}
E_x\left[e^{A_{t-u}^{\nu}};|B_{t-u}|>R(t)\right]
\leq c_2 P_0(|B_1|>R(t)-M)\leq c_3 e^{-(R(t)-M)^2/2}R(t)^{d-2}.
\end{equation*}
Hence for any $x\in {\mathbb R}^d$, large $t\geq 1$ and $s\in [0, t-1]$, 
\begin{equation}\label{eq:(ii)-1}
{\rm (V)}
\leq c_4 E_x\left[\int_{t-s-1}^{t-s}e^{A_u^{\nu}}\,{\rm d}A_u^{\mu}\right]e^{-(R(t)-M)^2}R(t)^{2(d-2)}
\leq c_4 E_x\left[e^{A_t^{\nu^{+}}}A_t^{\mu}\right]e^{-(R(t)-M)^2}R(t)^{2(d-2)}.
\end{equation}
By \cite[p.73, Corollary to Proposition 3.8]{CZ95}, 
we have for any $x\in {\mathbb R}^d$ and large $t$,
\begin{equation*}
E_x\left[e^{A_t^{\nu^{+}}}A_t^{\mu}\right]
\leq E_x\left[e^{2A_t^{\nu^{+}}}\right]^{1/2}E_x\left[(A_t^{\mu})^2\right]^{1/2}
\leq e^{c_5 t}.
\end{equation*}
Then \eqref{eq:asymp-eta} and \eqref{eq:(ii)-1} yield that 
for any $x\in K$, large $t\geq 1$ and $s\in[0,t-1]$, 
$${\rm (V)}\leq c_4 e^{c_5 t}e^{-(R(t)-M)^2}R(t)^{2(d-2)}
\leq c_6 \left(e^{\lambda s}\eta(t)\right)^2.$$
Combining this with \eqref{eq:second-moment-1}, 
we arrive at the desired conclusion.
\end{proof}

\subsection{Uniform estimate for the critical case}
Recall from \eqref{eq:rad-br} that for $\kappa\in {\mathbb R}$, 
\begin{equation*}
R_1^{(\kappa)}(t)=\sqrt{\frac{-\lambda}{2}}t+\frac{d-1}{2\sqrt{-2\lambda}}\log t+\kappa.
\end{equation*}
In this subsection, we clarify how $\kappa$ depends on  
the inequalities in Lemmas \ref{lem:exp-bound} and \ref{lem:second-moment}.
Here we note that if $R(t)=R_1^{(\kappa)}(t)$, then \eqref{eq:asymp-eta} becomes  
\begin{equation}\label{eq:asymp-eta-1}
\eta(t)\sim c_*e^{-\sqrt{-2\lambda}\kappa} \ (t\rightarrow\infty)
\end{equation}
for 
\begin{equation}\label{eq:const-1}
c_*=c_d\left(\sqrt{\frac{-\lambda}{2}}\right)^{(d-1)/2}
=\frac{(-\lambda)^{(d-3)/2}}{2(2\pi)^{(d-1)/2}}
\int_{{\mathbb R}^d}\left(\int_{S^{d-1}}e^{\sqrt{-2\lambda}\langle \theta,z\rangle}\,{\rm d}\theta\right)
h(z)\,\nu({\rm d}z).
\end{equation}

\begin{lem}\label{lem:uniform-kappa-1}
Let $\nu^{+}$ and $\nu^{-}$ be Kato class measures with compact support in ${\mathbb R}^d$ 
such that $\lambda<0$. 
Then for any compact set $K$ in ${\mathbb R}^d$ and $\kappa>0$,
\begin{equation*}
\lim_{t\rightarrow\infty}\sup_{x\in K}
\left|\frac{e^{\sqrt{-2\lambda}\kappa}}{h(x)}E_x\left[e^{A_t^{\nu}};|B_t|>R_1^{(\kappa)}(t)\right]-c_*\right|=0.
\end{equation*}
\end{lem}

\begin{proof}
We write $R(t)=R_1^{(\kappa)}(t)$ for simplicity. 
Let $K$ be a compact set in ${\mathbb R}^d$ and $\kappa>0$. 
Then by Lemma \ref{lem:exp-bound}, 
there exist positive constants $c_1$, $c_2$, $T$ such that 
for all $t\geq T$ and $x\in K$,
\begin{equation*}
\left|E_x\left[e^{A_t^{\nu}};|B_t|>R(t)\right]-h(x)\eta(t)\right|
\leq c_1e^{-c_2t}t^{-(d-1)/2}
e^{-\sqrt{-2\lambda}\kappa}.
\end{equation*}
This implies that 
\begin{equation*}
\begin{split}
&\left|\frac{e^{\sqrt{-2\lambda}\kappa}}{h(x)}E_x\left[e^{A_t^{\nu}};|B_t|>R(t)\right]-c_*\right|\\
&\leq \left|\frac{e^{\sqrt{-2\lambda}\kappa}}{h(x)}E_x\left[e^{A_t^{\nu}};|B_t|>R(t)\right]
-e^{\sqrt{-2\lambda}\kappa}\eta(t)\right|+|e^{\sqrt{-2\lambda}\kappa}\eta(t)-c_*|\\
&\leq \frac{c_1e^{-c_2t}}{\inf_{y\in K}h(y)}t^{-(d-1)/2}+|e^{\sqrt{-2\lambda}\kappa}\eta(t)-c_*|.
\end{split}
\end{equation*}
As $t\rightarrow\infty$, the right hand side goes to $0$ by \eqref{eq:asymp-eta-1}, 
which completes the proof.
\end{proof}

\begin{lem}\label{lem:uniform-kappa-2}
Let $\nu^{+}$ and $\nu^{-}$ be as in Lemma {\rm \ref{lem:uniform-kappa-1}}.
Let $K$ be a compact set in ${\mathbb R}^d$ 
and $\mu$ a Kato class measure with compact support in ${\mathbb R}^d$. 
\begin{enumerate}
\item[{\rm (i)}] For any $\alpha\in (0,1/2)$, 
there exist $C>0$, and $T=T(\kappa)\geq 1$ for any $\kappa>0$, 
such that for any $x\in K$, $t\geq T$ and $s\in [0,t-1]$,
\begin{equation*}
E_x\left[\int_0^{t-s}e^{A_u^{\nu}}
E_{B_u}\left[e^{A_{t-s-u}^{\nu}};|B_{t-s-u}|>R_1^{(\kappa)}(t)\right]^2\,{\rm d}A_u^{\mu}\right]
\leq Ce^{2c_{\star}\kappa}\left(e^{\lambda s}\eta(t)\right)^2
\end{equation*}
with a positive constant $c_{\star}=\sqrt{-2\lambda}-\sqrt{-\lambda/2}/(1-\alpha)$.
\item[{\rm (ii)}] The next equality holds.
\begin{equation*}
\lim_{\kappa\rightarrow\infty}\limsup_{t\rightarrow\infty}\sup_{x\in K}
e^{\sqrt{-2\lambda}\kappa}E_x\left[\int_0^{t}e^{A_s^{\nu}}
E_{B_s}\left[e^{A_{t-s}^{\nu}};|B_{t-s}|>R_1^{(\kappa)}(t)\right]^2\,{\rm d}A_s^{\mu}\right]
=0.
\end{equation*}
\end{enumerate}
\end{lem}

\begin{proof}
We write $R(t)=R_1^{(\kappa)}(t)$ for simplicity. We first prove (i). 
Fix $\alpha\in (0,1/2)$ and 
let $M>0$ satisfy $K\cup \supp[|\nu|]\cup \supp[\mu]\subset B_0(M)$. 
Then in a similar way to \eqref{eq:upper-bm-0}, 
we can take $c_1>0$ so that 
for any $\kappa>0$, there exists $T_1=T_1(\kappa)\geq 1$ such that 
for all $t\geq T_1$, $s\in [0,t-1]$ and $y\in \overline{B_0(M)}$,
\begin{equation}\label{eq:upper-bm-00}
\begin{split}
&P_y(|B_{t-s}|>R(t))
\leq P_0(|B_{t-s}|>R(t)-M)\\
&\leq c_1e^{\lambda s} \exp\left(-\lambda s-\frac{(R(t)-M)^2}{2(t-s)}\right)
\left(\frac{R(t)-M}{\sqrt{t-s}}\right)^{d-2}=c_1e^{\lambda s}{\rm (VI)}.
\end{split}
\end{equation}
We will prove later that there exists $c_2>0 $ 
such that for any $\kappa>0$, 
we can take $T_2=T_2(\kappa)\geq 1$ so that   
for all $t\geq T_2$, $s\in [0, t-1]$ and  $y\in \overline{B_0(M)}$, 
\begin{equation}\label{eq:upper-bm-01}
{\rm (VI)}
\leq c_2e^{c_{\star}\kappa}e^{-\lambda t-\sqrt{-2\lambda}R(t)}R(t)^{(d-2)/2}.
\end{equation} 
Then by \eqref{eq:upper-bm-00}, 
we can take $c_3>0$ so that 
for any $\kappa>0$, there exists $T_3=T_3(\kappa)\geq 1$ such that 
for all $t\geq T_3$, $s\in [0,t-1]$ and $y\in \overline{B_0(M)}$,
\begin{equation}\label{eq:upper-bm-02}
\begin{split}
P_y(|B_{t-s}|>R(t))\leq P_0(|B_{t-s}|>R(t)-M)
\leq c_3e^{c_{\star}\kappa}
e^{\lambda s}e^{-\lambda t-\sqrt{-2\lambda}R(t)}t^{(d-2)/2}.
\end{split}
\end{equation}
The proof of (i) is then complete 
by replacing \eqref{eq:int-upper-s} with \eqref{eq:upper-bm-02} 
and then following the proof of Lemma \ref{lem:second-moment}.

We are now in a position to prove \eqref{eq:upper-bm-01}. 
We first suppose that $s\in [0,\alpha t]$. 
In the same way as for deducing \eqref{eq:ineq-exp-1} and \eqref{eq:ineq-exp-2}, 
there exist $c_4>0$ and $T_4=T_4(\kappa)\geq 1$ such that for all $t\geq T_4$,
\begin{equation*}
-\lambda s-\frac{(R(t)-M)^2}{2(t-s)}
\leq -\lambda \alpha t-\frac{(R(t)-M)^2}{2(1-\alpha)t}
\leq c_{\star}\kappa
+(-\lambda) t-\sqrt{-2\lambda}R(t)-c_4t.
\end{equation*}
Hence there exists $c_5>0$, 
which is independent of $\kappa$, such that 
for any $s\in [0,\alpha t]$ and $t\geq T_4$, 
\begin{equation}\label{eq:upper-bm-a}
{\rm (VI)}
\leq 
c_5 e^{c_{\star}\kappa}e^{-c_4 t}e^{-\lambda t-\sqrt{-2\lambda}R(t)}t^{(d-2)/2}.
\end{equation}

We next suppose that $s\in [\alpha t, \beta t]$ for a fixed constant $\beta\in (1/2,1)$. 
If $t$ is so large that $t-(R(t)-M)/\sqrt{-2\lambda}\in [\alpha t,\beta t]$, 
then the function 
\begin{equation*}
f(s)=-\lambda s-\frac{(R(t)-M)^2}{2(t-s)} \ (s\in [\alpha t,\beta t])
\end{equation*} attains 
a maximal value at $s_0=t-(R(t)-M)/\sqrt{-2\lambda}$ and 
\begin{equation*}
f(s_0)=-\lambda t-\sqrt{-2\lambda}R(t)+\sqrt{-2\lambda}M.
\end{equation*}
Therefore, there exists $c_6>0$ such that for any $\kappa>0$,  
we can take $T_5=T_5(\kappa)\geq 1$ so that 
for all $t\geq T_5$, $s\in [\alpha t,\beta t]$ and $y\in \overline{B_0(M)}$,
\begin{equation}\label{eq:upper-bm-b}
{\rm (VI)} \leq c_6 e^{-\lambda t-\sqrt{-2\lambda}R(t)}t^{(d-2)/2}.
\end{equation}

For all sufficiently large $t$ and $s\in [\beta t,t-1]$, 
we can also show that there exist positive constants $c_7$ and $c_8$,  
which are independent of $\kappa$, such that 
\begin{equation*}
{\rm (VI)}
\leq 
c_7 e^{c_{\star}\kappa}e^{-c_8 t}e^{-\lambda t-\sqrt{-2\lambda}R(t)}t^{d-2}.
\end{equation*}
Combining this with \eqref{eq:upper-bm-a} and \eqref{eq:upper-bm-b}, 
we obtain \eqref{eq:upper-bm-01} so that the proof of (i) is complete. 

We next prove (ii). 
For any $\alpha\in (0,1/2)$, 
there exists $c_9>0$ by (i) such that for any $\kappa>0$, 
we can take $T_6=T_6(\kappa)\geq 1$ such that  
for all $t\geq T_6$ and $x\in K$,
\begin{equation*}
\begin{split}
&e^{\sqrt{-2\lambda}\kappa}E_x\left[\int_0^{t}e^{A_s^{\nu}}
E_{B_s}\left[e^{A_{t-s}^{\nu}};|B_{t-s}|>R(t)\right]^2\,{\rm d}A_s^{\mu}\right]\\
&\leq c_9e^{2c_{\star}\kappa}e^{\sqrt{-2\lambda}\kappa}\eta(t)^2
=c_9\exp(-\sqrt{-2\lambda}\kappa\alpha/(1-\alpha))(e^{\sqrt{-2\lambda}\kappa}\eta(t))^2.
\end{split}
\end{equation*}
Then by \eqref{eq:asymp-eta-1}, 
the last expression above goes to $0$ 
by letting $t\rightarrow\infty$ and then $\kappa\rightarrow\infty$.
This completes the proof of (ii). 
\end{proof}

\section{Proof of Theorem \ref{thm:lim-dist}}
Throughout this section, 
we impose Assumption \ref{assum:b} on 
the branching rate $\mu$ and branching mechanism ${\mathbf p}$. 
Then Lemmas \ref{lem:uniform-kappa-1} and \ref{lem:uniform-kappa-2} are valid 
by taking $\nu=(Q-1)\mu$ and $R(t)=R_1^{(\kappa)}(t)$ in \eqref{eq:rad-br}. 
In what follows, we assume that $c_*$ 
is the constant in \eqref{eq:const-1} with $\nu=(Q-1)\mu$.

As in \cite{CH14}, we first reveal the limiting behavior of ${\mathbf P}_x(L_t>R_1^{(\kappa)}(t))$
which is uniform on each compact set.  
We also show a similar behavior for the expectation related to the martingale $M_t$, 
which is motivated by \cite[Lemma 8]{B20}.  

\begin{lem}\label{lem:conv-upper-0}
Let $K$ be a compact set in ${\mathbb R}^d$. Then
\begin{equation}\label{eq:conv-upper-0-0}
\lim_{\kappa\rightarrow\infty}\limsup_{t\rightarrow\infty}\sup_{x\in K}
\left|\frac{e^{\sqrt{-2\lambda}\kappa}}{h(x)}{\mathbf P}_x(L_t>R_1^{(\kappa)}(t))-c_*\right|=0
\end{equation}
and 
\begin{equation}\label{eq:conv-upper-0-1}
\lim_{\gamma\rightarrow +0}\limsup_{t\rightarrow\infty}\sup_{x\in K}
\left|\frac{1}{\gamma h(x)}{\mathbf E}_x\left[1-e^{-\gamma c_* M_t}\right]-c_*\right|=0.
\end{equation}
\end{lem}

\begin{proof}
We first prove \eqref{eq:conv-upper-0-0}. 
For $R\geq 0$, let $Z_t^R$ be the total number of particles at time $t$ 
on the set $\{y\in {\mathbb R}^d \mid |y|>R\}$. 
For any $x\in {\mathbb R}^d$ and $t\geq 0$, since  
$${\mathbf P}_x(L_t>R)={\mathbf P}_x\left(Z_t^{R}\geq 1\right),$$
we have by the Paley-Zygmund and Chebyshev inequalities, 
\begin{equation}\label{eq:moment}
\frac{{\mathbf E}_x\left[Z_t^{R}\right]^2}{{\mathbf E}_x\left[(Z_t^{R})^2\right]}
\leq {\mathbf P}_x(L_t>R)\leq {\mathbf E}_x\left[Z_t^{R}\right].
\end{equation}

Let $R(t)=R_1^{(\kappa)}(t)$. 
Since we see by \eqref{eq:moment} and Lemma \ref{lem:moment} (i) that  
\begin{equation*}
{\mathbf P}_x(L_t>R(t))\leq {\mathbf E}_x\left[Z_t^{R(t)}\right]
=E_x\left[e^{A_t^{(Q-1)\mu}};|B_t|>R(t)\right],
\end{equation*}
Lemma \ref{lem:uniform-kappa-1} implies that 
for any compact set $K$ in ${\mathbb R}^d$ and $\kappa>0$, 
\begin{equation}\label{eq:uniform-first-moment}
\limsup_{t\rightarrow\infty}\sup_{x\in K}
\frac{e^{\sqrt{-2\lambda}\kappa}}{h(x)}{\mathbf P}_x(L_t>R(t))
\leq \lim_{t\rightarrow\infty}\sup_{x\in K}
\frac{e^{\sqrt{-2\lambda}\kappa}}{h(x)}E_x\left[e^{A_t^{(Q-1)\mu}};|B_t|>R(t)\right]=c_*.
\end{equation}

By Lemma \ref{lem:moment} (ii), 
\begin{equation*}
\begin{split}
{\mathbf E}_x\left[(Z_t^{R(t)})^2\right]
&=E_x\left[e^{A_t^{(Q-1)\mu}};|B_t|>R(t)\right]\\
&+E_x\left[\int_0^t e^{A_s^{(Q-1)\mu}}E_{B_s}\left[e^{A_{t-s}^{(Q-1)\mu}};|B_{t-s}|>R(t)\right]^2
\,{\rm d}A_s^{\nu_R}\right].
\end{split}
\end{equation*}
Therefore, we have by Lemmas \ref{lem:uniform-kappa-1} and \ref{lem:uniform-kappa-2}, 
\begin{equation}\label{eq:uniform-second-moment}
\limsup_{\kappa\rightarrow\infty}
\limsup_{t\rightarrow\infty}\sup_{x\in K}\frac{e^{\sqrt{-2\lambda}\kappa}}{h(x)}
{\mathbf E}_x\left[(Z_t^{R(t)})^2\right]\leq c_*.
\end{equation}
Lemmas \ref{lem:moment} (i) and \ref{lem:uniform-kappa-1} also imply 
that for any compact set $K$ in ${\mathbb R}^d$ and $\kappa>0$,
\begin{equation}\label{eq:uniform-first-moment-1}
\lim_{t\rightarrow\infty}\inf_{x\in K}
\frac{e^{\sqrt{-2\lambda}\kappa}}{h(x)}{\mathbf E}_x\left[Z_t^{R(t)}\right]
=\lim_{t\rightarrow\infty}\inf_{x\in K}
\frac{e^{\sqrt{-2\lambda}\kappa}}{h(x)}E_x\left[e^{A_t^{(Q-1)\mu}};|B_t|> R(t)\right]=c_*.
\end{equation}
Then, since \eqref{eq:moment} yields that for any $\kappa>0$ and $t>0$,
\begin{equation*}
\begin{split}
\inf_{x\in K}\frac{e^{\sqrt{-2\lambda}\kappa}}{h(x)} {\mathbf P}_x(L_t>R(t))
\geq 
\left(\inf_{x\in K}\frac{e^{\sqrt{-2\lambda}\kappa}}{h(x)} {\mathbf E}_x\left[Z_t^{R(t)}\right]\right)^2
\left(\sup_{x\in K}\frac{e^{\sqrt{-2\lambda}\kappa}}{h(x)} {\mathbf E}_x\left[(Z_t^{R(t)})^2\right]\right)^{-1},
\end{split}
\end{equation*}
we have by \eqref{eq:uniform-second-moment} 
and \eqref{eq:uniform-first-moment-1},
\begin{equation*}
\begin{split}
\liminf_{\kappa\rightarrow\infty}\liminf_{t\rightarrow\infty}\inf_{x\in K}
\frac{e^{\sqrt{-2\lambda}\kappa}}{h(x)}{\mathbf P}_x(L_t>R(t))
\geq c_*.
\end{split}
\end{equation*}
Combining this with \eqref{eq:uniform-first-moment}, 
we get \eqref{eq:conv-upper-0-0}. 

We next prove \eqref{eq:conv-upper-0-1}. 
By the Taylor theorem, there exists a random variable $\theta\in (0,1)$ for any $\gamma>0$ 
such that 
\begin{equation*}
\frac{1-e^{-c_*\gamma M_t}}{\gamma}=c_*M_t-\frac{e^{-\theta c_*\gamma M_t}}{2}c_*^2\gamma M_t^2.
\end{equation*}
Taking the expectation in this equality, we have by \eqref{eq:mtg-moment-1}, 
\begin{equation}\label{eq:mtg-expectation}
\frac{{\mathbf E}_x\left[1-e^{-c_*\gamma M_t}\right]}{\gamma}
=c_*h(x)-\frac{c_*^2\gamma}{2}
{\mathbf E}_x\left[e^{-\theta c_*\gamma M_t}M_t^2\right].
\end{equation}
By \eqref{eq:mtg-moment-2},
\begin{equation}\label{eq:mtg-upper}
\begin{split}
{\mathbf E}_x\left[e^{-\theta c_*\gamma M_t}M_t^2\right]
&\leq {\mathbf E}_x\left[M_t^2\right]\\
&=e^{2\lambda t}E_x\left[e^{A_t^{(Q-1)\mu}}h(B_t)^2\right]
+E_x\left[\int_0^te^{2\lambda s+A_s^{(Q-1)\mu}}h(B_s)^2\,{\rm d}A_s^{\nu_R}\right].
\end{split}
\end{equation}
We then obtain by \eqref{eq:eigen-eq} and \eqref{eq:gaugeable}, 
respectively,  
\begin{equation*}
e^{2\lambda t}E_x\left[e^{A_t^{(Q-1)\mu}}h(B_t)^2\right]
\leq e^{2\lambda t}E_x\left[e^{A_t^{(Q-1)\mu}}h(B_t)\right]\|h\|_{\infty}=e^{\lambda t}h(x)\|h\|_{\infty}
\leq \|h\|_{\infty}^2
\end{equation*}
and 
\begin{equation*}
E_x\left[\int_0^te^{2\lambda s+A_s^{(Q-1)\mu}}h(B_s)^2\,{\rm d}A_s^{\nu_R}\right]
\leq \|h\|_{\infty}^2
\sup_{y\in {\mathbb R}^d}
E_y\left[\int_0^{\infty}e^{2\lambda s+A_s^{(Q-1)\mu}}\,{\rm d}A_s^{\nu_R}\right]<\infty.
\end{equation*}
Hence by \eqref{eq:mtg-upper},
\begin{equation*}
\sup_{x\in {\mathbb R}^d, t\geq 0}{\mathbf E}_x\left[e^{-\theta c_*\gamma M_t}M_t^2\right]<\infty.
\end{equation*}
By combining this with \eqref{eq:mtg-expectation}, 
there exists $c>0$ such that for any $x\in K$,
\begin{equation*}
\left|\frac{1}{\gamma h(x)}{\mathbf E}_x\left[1-e^{-c_*\gamma M_t}\right]-c_*\right|
=\frac{c_*^2\gamma}{2 h(x)}
{\mathbf E}_x\left[e^{-\theta c_*\gamma M_t}M_t^2\right]
\leq \frac{c\gamma}{\inf_{y\in K}h(y)}.
\end{equation*}
We complete the proof by letting $t\rightarrow\infty$ and then $\gamma\rightarrow +0$. 
\end{proof}

Let ${\cal L}$ be the totality of compact sets in ${\mathbb R}^d$. 
We next show a uniform version of  Lemma \ref{lem:conv-upper-0}. 
\begin{prop}\label{prop:conv-upper-0}
The following assertions hold{\rm :}
\begin{equation}\label{eq:conv-upper-1-0}
\lim_{\kappa\rightarrow\infty}\sup_{L\in {\cal L}}
\limsup_{t\rightarrow\infty}\sup_{x\in L}
\left|\frac{e^{\sqrt{-2\lambda}\kappa}}{h(x)}{\mathbf P}_x(L_t>R_1^{(\kappa)}(t))-c_*\right|=0
\end{equation}
and 
\begin{equation}\label{eq:conv-upper-1-1}
\lim_{\gamma\rightarrow +0}\sup_{L\in {\cal L}}
\limsup_{t\rightarrow\infty}\sup_{x\in L}
\left|\frac{1}{\gamma h(x)}{\mathbf E}_x\left[1-e^{-c_*\gamma M_t}\right]-c_*\right|=0.
\end{equation}
\end{prop}

\begin{proof}
Here we prove \eqref{eq:conv-upper-1-0} only 
because \eqref{eq:conv-upper-1-1} follows in the same way. 
Let $R(t)=R_1^{(\kappa)}(t)$. 
Since we see by \eqref{eq:uniform-first-moment} that 
\begin{equation*}
\limsup_{\kappa\rightarrow\infty}
\sup_{L\in {\cal L}}\limsup_{t\rightarrow\infty}
\sup_{x\in L}\frac{e^{\sqrt{-2\lambda}\kappa}}{h(x)}{\mathbf P}_x(L_t>R(t))
\leq c_*, 
\end{equation*}
it is sufficient to show that 
\begin{equation}\label{eq:mtg-lower}
\liminf_{\kappa\rightarrow\infty}
\inf_{L\in {\cal L}}\liminf_{t\rightarrow\infty}
\inf_{x\in L}\frac{e^{\sqrt{-2\lambda}\kappa}}{h(x)}{\mathbf P}_x(L_t>R(t))
\geq c_*.
\end{equation}

By Lemma \ref{lem:conv-upper-0}, 
there exists $\kappa_0=\kappa_0(\varepsilon, K)$ 
for any $\varepsilon>0$ and $K\in {\cal L}$ such that 
we can take $T_0=T_0(\varepsilon,K,\kappa)>0$ for any $\kappa\geq \kappa_0$ 
so that for all $t\geq T_0$ and $x\in K$,
\begin{equation*}
|e^{\sqrt{-2\lambda}\kappa}{\mathbf P}_x(L_t>R(t))-c_*h(x)|\leq \varepsilon h(x).
\end{equation*}
Note that for any $t>0$ and $s\in [0,t)$, 
\begin{equation*}
R(t)=\sqrt{\frac{-\lambda}{2}}(t-s)+\frac{d-1}{2\sqrt{-2\lambda}}\log (t-s)
+\sqrt{\frac{-\lambda}{2}}s+\frac{d-1}{2\sqrt{-2\lambda}}\log \left(\frac{t}{t-s}\right)+\kappa
\end{equation*}
and 
\begin{equation*}
\sqrt{\frac{-\lambda}{2}}s+\frac{d-1}{2\sqrt{-2\lambda}}\log \left(\frac{t}{t-s}\right)+\kappa
\geq \kappa.
\end{equation*}
Hence for each fixed $T>0$, 
we have for all $s\in [0,T]$, $t\geq T_0+T$ and $x\in K$, 
\begin{equation}\label{eq:local-conv}
|e^{\sqrt{-2\lambda}\kappa}{\mathbf P}_x(L_{t-s}>R(t))-c_*h(x)|\leq \varepsilon h(x).
\end{equation}

In what follows, let $K\in {\cal L}$ satisfy $K\supset \supp[\mu]$,
and let $\sigma$ be the hitting time of some particle to $K$. 
Since particles reproduce only on $\supp[\mu]$, 
$\sigma$ is relevant only to the initial particle; 
if $x\in K$, then ${\mathbf P}_x(\sigma=0)=1$. 
Otherwise, no branching occurs until the initial particle hits $K$. 
We use the same notation $\sigma$ also 
as the hitting time of the Brownian motion to $K$. 

For $T>0$ and $t>T$, we have by the strong Markov property,  
\begin{equation}\label{eq:lower-dist-1}
{\mathbf P}_x(L_t>R(t))\geq {\mathbf P}_x(L_t>R(t), \sigma\leq T)
=E_x\left[{\mathbf P}_{B_{\sigma}}(L_{t-s}>R(t))|_{s=\sigma}; \sigma\leq T\right].
\end{equation}
Then by \eqref{eq:local-conv}, 
there exists $\kappa_0=\kappa_0(\varepsilon)$ for any $\varepsilon>0$ such that 
we can take $T_1=T_1(\varepsilon, \kappa)>0$ for any $\kappa\geq \kappa_0$ so that 
for all $t\geq T_1+T$ and $x\in {\mathbb R}^d$,
\begin{equation}\label{eq:lower-dist-2}
\begin{split}
E_x\left[{\mathbf P}_{B_{\sigma}}(L_{t-s}>R(t))|_{s=\sigma}; \sigma\leq T\right]
&\geq (c_*-\varepsilon)e^{-\sqrt{-2\lambda}\kappa}E_x\left[h(B_{\sigma}); \sigma\leq T\right]\\
&\geq  (c_*-\varepsilon)e^{-\sqrt{-2\lambda}\kappa}
E_x\left[e^{\lambda\sigma}h(B_{\sigma}); \sigma\leq T\right].
\end{split}
\end{equation}

Since $P_x\left(A_{\sigma\wedge T}^{(Q-1)\mu}=0\right)=1$ 
and $e^{\lambda t+A_t^{(Q-1)\mu}}h(B_t)$ is a $P_x$-martingale 
by noting that $e^{\lambda t}p_t^{(Q-1)\mu}h=h$, 
the optional stopping theorem yields that 
\begin{equation*}
E_x\left[e^{\lambda(\sigma\wedge T)}h(B_{\sigma\wedge T})\right]
=E_x\left[e^{\lambda(\sigma\wedge T)+A_{\sigma\wedge T}^{(Q-1)\mu}}h(B_{\sigma\wedge T})\right]
=h(x).
\end{equation*}
We also see that
\begin{equation*}
e^{\lambda T}E_x\left[h(B_T); T<\sigma\right]\leq e^{\lambda T}\|h\|_{\infty}
\end{equation*}
and thus 
\begin{equation*}
\begin{split}
&E_x\left[e^{\lambda\sigma}h(B_{\sigma}); \sigma\leq T\right]
=E_x\left[e^{\lambda(\sigma\wedge T)}h(B_{\sigma\wedge T}); \sigma\leq T\right]\\
&=E_x\left[e^{\lambda(\sigma\wedge T)}h(B_{\sigma\wedge T})\right]
-e^{\lambda T}E_x\left[h(B_T); T<\sigma\right]
\geq h(x)-e^{\lambda T}\|h\|_{\infty}.
\end{split}
\end{equation*}
Combining this with \eqref{eq:lower-dist-1} and \eqref{eq:lower-dist-2}, 
we get for all $t\geq T_1+T$ and $x\in {\mathbb R}^d$, 
\begin{equation}\label{eq:lower-dist-4}
e^{\sqrt{-2\lambda}\kappa}{\mathbf P}_x(L_t>R(t))
\geq (c_*-\varepsilon)\left(h(x)-e^{\lambda T}\|h\|_{\infty}\right).
\end{equation}

For any $L\in {\cal L}$ and $t\geq T_1+T$, we have by \eqref{eq:lower-dist-4},
\begin{equation*}
\inf_{x\in L}\frac{e^{\sqrt{-2\lambda}\kappa}}{h(x)}{\mathbf P}_x(L_t>R(t))
\geq (c_*-\varepsilon)\left(1-\frac{e^{\lambda T}\|h\|_{\infty}}{\inf_{x\in L}h(x)}\right).
\end{equation*}
Hence we obtain by letting  $t\rightarrow\infty$ and then $T\rightarrow\infty$, 
\begin{equation*}
\liminf_{t\rightarrow\infty}
\inf_{x\in L}\frac{e^{\sqrt{-2\lambda}\kappa}}{h(x)}{\mathbf P}_x(L_t>R(t))
\geq c_*-\varepsilon.
\end{equation*}
Since the right hand side above is independent of $L\in {\cal L}$, 
we get 
\begin{equation*}
\inf_{L\in {\cal L}}\liminf_{t\rightarrow\infty}
\inf_{x\in L}\frac{e^{\sqrt{-2\lambda}\kappa}}{h(x)}{\mathbf P}_x(L_t>R(t))
\geq c_*-\varepsilon.
\end{equation*}
We finally arrive at \eqref{eq:mtg-lower} by letting $\kappa\rightarrow\infty$ and then $\varepsilon\rightarrow 0$. 
\end{proof}

\begin{proof}[Proof of Theorem {\rm \ref{thm:lim-dist}}.] 
Since ${\mathbf P}_x(L_t<\infty)=1$ for any $t\geq 0$, 
there exists $r_1=r_1(\varepsilon,T_1)>0$ 
for any $\varepsilon\in (0,1)$ and $T_1>0$ such that 
\begin{equation}\label{eq:bound-1}
{\mathbf P}_x(L_{T_1}\leq r_1)\geq 1-\varepsilon.
\end{equation}
For any $t\geq T_1$, since 
\begin{equation*}
{\mathbf P}_x(L_t\leq R(t), L_{T_1}>r_1)
\leq {\mathbf P}_x(L_{T_1}>r_1)<\varepsilon, 
\end{equation*}
we have 
\begin{equation*}
\begin{split}
{\mathbf P}_x(L_t\leq R(t))
&={\mathbf P}_x(L_t\leq R(t), L_{T_1}\leq r_1)+{\mathbf P}_x(L_t\leq R(t), L_{T_1}>r_1)\\
&\leq {\mathbf P}_x(L_t\leq R(t), L_{T_1}\leq r_1)+\varepsilon
\end{split}
\end{equation*}
so that 
\begin{equation}\label{eq:vii-1}
\begin{split}
&{\mathbf P}_x\left(L_t>R(t)\right) 
-{\mathbf E}_x\left[1-\exp\left(-c_*e^{-\sqrt{-2\lambda}\kappa}M_t\right)\right]\\
&={\mathbf E}_x\left[\exp\left(-c_*e^{-\sqrt{-2\lambda}\kappa}M_t\right)\right]
-{\mathbf P}_x\left(L_t\leq R(t)\right) \\
&\geq 
{\mathbf E}_x\left[\exp\left(-c_*e^{-\sqrt{-2\lambda}\kappa}M_t\right);L_{T_1}\leq r_1\right]
-{\mathbf P}_x(L_t\leq R(t), L_{T_1}\leq r_1)-\varepsilon\\
&=
{\mathbf E}_x\left[\left\{\exp\left(-c_*e^{-\sqrt{-2\lambda}\kappa}M_t\right)
-{\bf 1}_{\{L_t\leq R(t)\}}\right\};L_{T_1}\leq r_1\right]
-\varepsilon={\rm (VII)}-\varepsilon.
\end{split}
\end{equation}
Here we recall that $e_0=\inf\{t>0 \mid Z_t=0\}$ is the extinction time. 
If $T_1\geq e_0$, then for any $t\geq T_1$, we have $M_t=0$ and $L_t=0$ by definition. 
Therefore, the Markov property implies that 
\begin{equation}\label{eq:vii-2}
\begin{split}
&{\rm (VII)}
={\mathbf E}_x\left[\left\{\exp\left(-c_*e^{-\sqrt{-2\lambda}\kappa}M_t\right)
-{\bf 1}_{\{L_t\leq R(t)\}}\right\};T_1<e_0, L_{T_1}\leq r_1\right]\\
&={\mathbf E}_x\left[\left\{{\mathbf E}_{{\mathbf B}_{T_1}}
\left[\exp\left(-c_*e^{-\sqrt{-2\lambda}\kappa}M_{t-T_1}\right)\right]
-{\mathbf P}_{{\mathbf B}_{T_1}}\left(L_{t-T_1}\leq R(t)\right)\right\};T_1<e_0,L_{T_1}\leq r_1\right].
\end{split}
\end{equation}

Let $R(t)=R_1^{(\kappa)}(t)$. For $T>0$ and $t>T$, let 
\begin{equation*}
F_{\kappa}(t,T):=\sqrt{\frac{-\lambda}{2}}T+\frac{d-1}{2\sqrt{-2\lambda}}
\log \left(\frac{t}{t-T}\right)+\kappa.
\end{equation*}
Then
\begin{equation*}
R(t)=\sqrt{\frac{-\lambda}{2}}(t-T)+\frac{d-1}{2\sqrt{-2\lambda}}\log (t-T)
+F_{\kappa}(t,T)
\end{equation*}
and 
\begin{equation*}
F_{\kappa}(t,T)\geq \sqrt{\frac{-\lambda}{2}}T+\kappa.
\end{equation*}
Hence by Proposition \ref{prop:conv-upper-0},
there exists $\kappa_0=\kappa_0(\delta)>0$ for any $\delta\in (0,c_*)$ such that 
if $\sqrt{-\lambda/2}T+\kappa\geq \kappa_0$, then 
\begin{equation*}
\sup_{L\in {\cal L}}
\limsup_{t\rightarrow\infty}\sup_{x\in L}
\left|\frac{e^{\sqrt{-2\lambda}F_{\kappa}(t,T)}}{h(x)}{\mathbf P}_x(L_{t-T}>R(t))-c_*\right|<\delta.
\end{equation*}

We take $T_1=T_1(\delta)>0$ so that $\sqrt{-\lambda/2}T_1+\kappa\geq \kappa_0$. 
Since $\overline{B_0(r_1)}\in {\cal L}$, 
there exists $T_2=T_2(\delta,\varepsilon, T_1)>0$ such that 
for any $t\geq T_1+T_2$ and $y\in \overline{B_0(r_1)}$,
\begin{equation}\label{eq:upper-tail-0}
\begin{split}
(c_*-\delta)e^{-\sqrt{-2\lambda}F_{\kappa}(t,T_1)}h(y)
\leq {\mathbf P}_y(L_{t-T_1}>R(t))
&\leq (c_*+\delta) e^{-\sqrt{-2\lambda}F_{\kappa}(t,T_1)}h(y)\\
&\leq (c_*+\delta) e^{-\sqrt{-2\lambda}\kappa+\lambda T_1} h(y).
\end{split}
\end{equation}

Note that $1-r\leq e^{-r}$ for any $r\in {\mathbb R}$ 
and there exists $r_0=r_0(\delta)>0$ such that 
$1-r\geq e^{-(1+\delta)r}$ for any $r\in [0,r_0]$. 
Here we further take $T_1=T_1(\delta)>0$ so large that 
\begin{equation*}
(c_*+\delta) e^{-\sqrt{-2\lambda}\kappa+\lambda T_1}\|h\|_{\infty}\leq r_0.
\end{equation*}
Then for any $t\geq T_1+T_2$ and $y\in \overline{B_0(r_1)}$, 
we have by \eqref{eq:upper-tail-0}, 
\begin{equation}\label{eq:upper-tail}
\begin{split}
&\exp\left(-(1+\delta)(c_*+\delta)e^{-\sqrt{-2\lambda}\kappa+\lambda T_1}h(y)\right)\\
&\leq 1-{\mathbf P}_y(L_{t-T_1}>R(t))
\leq 
\exp\left(-(c_*-\delta)e^{-\sqrt{-2\lambda}F_{\kappa}(t,T_1)}h(y)\right)\\
&=\exp\left(-(c_*-\delta)e^{-\sqrt{-2\lambda}\kappa+\lambda T_1}
\left(\frac{t-T_1}{t}\right)^{(d-1)/2}h(y)\right).
\end{split}
\end{equation}

Suppose that $T_1<e_0$ and $L_{T_1}\leq r_1$. 
Then $Z_{T_1}\geq 1$ and ${\mathbf B}_{T_1}^k\in \overline{B_0(r_1)}$ for any $k=1,\dots, Z_{T_1}$. 
Since
\begin{equation*}
{\mathbf P}_{{\mathbf B}_{T_1}}(L_{t-T_1}\leq R(t))
=\prod_{k=1}^{Z_{T_1}}{\mathbf P}_{{\mathbf B}_{T_1}^k}(L_{t-T_1}\leq R(t))
=\prod_{k=1}^{Z_{T_1}}(1-{\mathbf P}_{{\mathbf B}_{T_1}^k}(L_{t-T_1}>R(t)))
\end{equation*}
we obtain by \eqref{eq:upper-tail-0},
\begin{equation}\label{eq:upper-lower}
\begin{split}
\exp\left(-{(1+\delta)(c_*+\delta)e^{-\sqrt{-2\lambda}\kappa}M_{T_1}}\right)
&\leq {\mathbf P}_{{\mathbf B}_{T_1}}(L_{t-T_1}\leq R(t))\\
&\leq \exp\left(-(c_*-\delta)e^{-\sqrt{-2\lambda}\kappa}
\left(\frac{t-T_1}{t}\right)^{(d-1)/2}M_{T_1}\right).
\end{split}
\end{equation}
Since
\begin{equation*}
{\mathbf E}_{{\mathbf B}_{T_1}}\left[\exp\left(-c_*e^{-\sqrt{-2\lambda}\kappa} M_{t-T_1}\right)\right]
=\prod_{k=1}^{Z_{T_1}}{\mathbf E}_{{\mathbf B}_{T_1}^k}\left[\exp\left(-c_*e^{-\sqrt{-2\lambda}\kappa}M_{t-T_1}\right)\right],
\end{equation*}
we have by following the argument for \eqref{eq:upper-lower}, 
\begin{equation*}
\begin{split}
\exp\left(-{(1+\delta)(c_*+\delta)e^{-\sqrt{-2\lambda}\kappa}M_{T_1}}\right)
&\leq {\mathbf E}_{{\mathbf B}_{T_1}}\left[\exp\left(-c_*e^{-\sqrt{-2\lambda}\kappa} M_{t-T_1}\right)\right]\\
&\leq \exp\left(-{(c_*-\delta)e^{-\sqrt{-2\lambda}\kappa}\left(\frac{t-T_1}{t}\right)^{(d-1)/2}M_{T_1}}\right).
\end{split}
\end{equation*}
Hence by \eqref{eq:bound-1} and \eqref{eq:vii-2}, 
we see that for any $t\geq T_1+T_2$,  
\begin{equation*}
\begin{split}
{\rm (VII)}&\geq 
{\mathbf E}_x\left[\exp\left(-{(1+\delta)(c_*+\delta)e^{-\sqrt{-2\lambda}\kappa}M_{T_1}}\right)\right]\\
&-{\mathbf E}_x\left[\exp\left(-{(c_*-\delta)e^{-\sqrt{-2\lambda}\kappa}\left(\frac{t-T_1}{t}\right)^{(d-1)/2}M_{T_1}}\right)\right]
-\varepsilon.
\end{split}
\end{equation*}
The right hand side above goes to $0$ 
by letting $t\rightarrow\infty$, $T_1\rightarrow\infty$, 
$\delta\rightarrow 0$ and then $\varepsilon\rightarrow 0$, 
We therefore obtain by \eqref{eq:vii-1}, 
\begin{equation*}
\liminf_{t\rightarrow\infty}\left({\mathbf P}_x\left(L_t>R(t)\right) 
-{\mathbf E}_x\left[1-\exp\left(-c_*e^{-\sqrt{-2\lambda}\kappa}M_t\right)\right]\right)
\geq 0.
\end{equation*}
In the same way, we have
\begin{equation*}
\limsup_{t\rightarrow\infty}\left({\mathbf P}_x\left(L_t>R(t)\right) 
-{\mathbf E}_x\left[1-\exp\left(-c_*e^{-\sqrt{-2\lambda}\kappa}M_t\right)\right]\right)
\leq 0
\end{equation*}
so that the proof is complete. 
\end{proof}

\section{Proofs of Theorem \ref{thm:equiv} and Corollary \ref{thm:yaglom}}

Throughout this section, 
we impose Assumption \ref{assum:b} on 
the branching rate $\mu$ and branching mechanism ${\mathbf p}$. 
We write $R(t)$ for $R_2(t)$, 
$\delta \in (\sqrt{-\lambda/2}, \sqrt{-2\lambda})$ 
in \eqref{eq:rad-br-1} or for $R_3(t)$ in \eqref{eq:rad-br-2}.

\begin{proof}[Proof of Theorem {\rm \ref{thm:equiv}}.] 
Let $K$ be a compact set in ${\mathbb R}^d$.  
Then by \eqref{eq:moment}, 
\begin{equation}\label{eq:limsup}
\limsup_{t\rightarrow\infty}\sup_{x\in K}\frac{{\mathbf P}_x(L_t>R(t))}{{\mathbf E}_x[Z_t^{R(t)}]}\leq 1.
\end{equation}

By Lemma \ref{lem:second-moment} with $s=0$, 
there exist $c_1>0$ and $T_1\geq 1$ such that for any $x\in K$ and $t\geq T_1$, 
\begin{equation}\label{eq:second-upper}
E_x\left[\int_0^t e^{A_s^{(Q-1)\mu}}E_{B_s}\left[e^{A_{t-s}^{(Q-1)\mu}};|B_{t-s}|>R(t)\right]^2\,{\rm d}A_s^{\nu_R}\right]
\leq c_1\eta(t)^2. 
\end{equation}
By \eqref{eq:br-fk} and Lemma \ref{lem:exp-bound} with $s=0$, 
there exist positive constants $c_2$, $c_3$, $c_4$ and $T_2$ 
such that for any $x\in K$ and $t\geq T_2$, 
\begin{equation*}
\begin{split}
{\mathbf E}_x[Z_t^{R(t)}]
=E_x\left[e^{A_t^{(Q-1)\mu}};|B_t|>R(t)\right]
&\geq (h(x)-c_2e^{-c_3 t})\eta(t)\\
&\geq \left(\inf_{y\in K}h(y)-c_2e^{-c_3 t}\right)\eta(t)
\geq c_4\eta(t).
\end{split}
\end{equation*}
Hence by \eqref{eq:br-fk-1} and \eqref{eq:second-upper}, 
there exist $c_5>0$ and $T_3\geq 1$ such that 
for any $x\in K$ and $t\geq T_3$,
\begin{equation}\label{eq:second-first}
\begin{split}
{\mathbf E}_x[(Z_t^{R(t)})^2]
&={\mathbf E}_x[Z_t^{R(t)}]
+E_x\left[\int_0^t e^{A_s^{(Q-1)\mu}}E_{B_s}\left[e^{A_{t-s}^{(Q-1)\mu}};|B_{t-s}|>R(t)\right]^2\,{\rm d}A_s^{\nu_R}\right]\\
&\leq {\mathbf E}_x[Z_t^{R(t)}]+c_1\eta(t)^2
\leq {\mathbf E}_x[Z_t^{R(t)}]\left(1+c_5 \eta(t)\right).
\end{split}
\end{equation}
Then by \eqref{eq:moment},
\begin{equation*}
\frac{{\mathbf P}_x(L_t>R(t))}{{\mathbf E}_x[Z_t^{R(t)}]}
\geq \frac{{\mathbf E}_x[Z_t^{R(t)}]}{{\mathbf E}_x[(Z_t^{R(t)})^2]}
\geq \frac{1}{1+c_5\eta(t)}.
\end{equation*}
Noting that $\eta(t)\rightarrow 0$ as $t\rightarrow\infty$ by \eqref{eq:asymp-eta}, we get
\begin{equation*}
\liminf_{t\rightarrow\infty}\inf_{x\in K}\frac{{\mathbf P}_x(L_t>R(t))}{{\mathbf E}_x[Z_t^{R(t)}]}\geq 1.
\end{equation*}
The proof is complete by combining this inequality with \eqref{eq:limsup}.
\end{proof}

\begin{proof}[Proof of Corollary  {\rm \ref{thm:yaglom}}.] 
Since $Z_t^{R(t)}$ takes a nonnegative value, 
we have
\begin{equation*}
{\mathbf P}_x(Z_t^{R(t)}>1)+{\mathbf P}_x(L_t>R(t))
={\mathbf P}_x(Z_t^{R(t)}>1)+{\mathbf P}_x(Z_t^{R(t)}\geq 1)
\leq {\mathbf E}_x[Z_t^{R(t)}]
\end{equation*}
so that 
\begin{equation*}
{\mathbf P}_x(Z_t^{R(t)}>1 \mid L_t>R(t))
\leq \frac{{\mathbf E}_x[Z_t^{R(t)}]}{{\mathbf P}_x(L_t>R(t))}-1.
\end{equation*}
Hence the proof is complete  by Theorem \ref{thm:equiv}.
\end{proof}

\appendix
\section{Appendix}
\subsection{Upper bounds of an integral in \eqref{eq:deri-upper}}
\label{appendix:int}

We use the same notations as in Lemma \ref{lem:iii}. 
We now give upper bounds of the integral at the last line of \eqref{eq:deri-upper}. 
For simplicity, we replace $R-M$ by $R$ and assume that $R>M$. 
Let
$$v=\sqrt{-\lambda}u-\frac{R}{\sqrt{2}u}.$$
Since
$$u=\frac{v+\sqrt{v^2+2\sqrt{-2\lambda}R}}{2\sqrt{-\lambda}}, \quad 
\frac{{\rm d}u}{u}=\frac{{\rm d}v}{\sqrt{v^2+2\sqrt{-2\lambda}R}},$$
we have
\begin{equation}\label{eq:minus-part-2}
\begin{split}
&\int_{\sqrt{t}}^{\infty}e^{-(\sqrt{-\lambda}u-R/(\sqrt{2}u))^2}\frac{1}{u^{d+1}}\,{\rm d}u\\
&=c_1\int_{(\sqrt{-2\lambda} t-R)/\sqrt{2t}}^{\infty}
e^{-v^2}\frac{1}{(v+\sqrt{v^2+2\sqrt{-2\lambda}R})^d}\frac{1}{\sqrt{v^2+2\sqrt{-2\lambda}R}}\,{\rm d}v.
\end{split}
\end{equation}

Suppose first that $R\geq \sqrt{-2\lambda}t$.
If $(\sqrt{-2\lambda}t-R)/\sqrt{2t}\leq v\leq 0$, then 
$$\frac{1}{v+\sqrt{v^2+2\sqrt{-2\lambda}R}}
=\frac{\sqrt{v^2+2\sqrt{-2\lambda}R}-v}{2\sqrt{-2\lambda}R}\leq 
\frac{\sqrt{v^2+2\sqrt{-2\lambda}R}}{\sqrt{-2\lambda}R}$$
and
$$(v^2+2\sqrt{-2\lambda}R)^{(d-1)/2}\leq c_2(|v|^{d-1}+R^{(d-1)/2}).$$
Therefore,
\begin{equation*}
\begin{split}
\frac{1}{(v+\sqrt{v^2+2\sqrt{-2\lambda}R})^d}\frac{1}{\sqrt{v^2+2\sqrt{-2\lambda}R}}
&\leq c_3\frac{(v^2+2\sqrt{-2\lambda}R)^{(d-1)/2}}{R^d}\\
&\leq \frac{c_4}{R^{(d+1)/2}}\left(\frac{|v|^{d-1}}{R^{(d-1)/2}}+1\right),
\end{split}
\end{equation*}
which implies that 
\begin{equation*}
\begin{split}
&\int_{(\sqrt{-2\lambda} t-R)/\sqrt{2t}}^0
e^{-v^2}\frac{1}{(v+\sqrt{v^2+2\sqrt{-2\lambda}R})^d}\frac{1}{\sqrt{v^2+2\sqrt{-2\lambda}R}}\,{\rm d}v\\
&\leq c_5R^{-(d+1)/2} \left(R^{-(d-1)/2}
\int_{(\sqrt{-2\lambda} t-R)/\sqrt{2t}}^0
e^{-v^2}|v|^{d-1}\,{\rm d}v
+\int_{(\sqrt{-2\lambda} t-R)/\sqrt{2t}}^0
e^{-v^2}\,{\rm d}v\right)\\
&\leq c_6R^{-(d+1)/2}. 
\end{split}
\end{equation*}
Since
$$\int_0^{\infty}
e^{-v^2}\frac{1}{(v+\sqrt{v^2+2\sqrt{-2\lambda}R})^d}\frac{1}{\sqrt{v^2+2\sqrt{-2\lambda}R}}\,{\rm d}v
\leq c_7R^{-(d+1)/2}\int_0^{\infty} e^{-v^2}\,{\rm d}v,$$
we obtain by \eqref{eq:minus-part-2},
\begin{equation}\label{eq:negative-upper}
\int_{\sqrt{t}}^{\infty}e^{-(\sqrt{-\lambda}u-R/(\sqrt{2}u))^2}\frac{1}{u^{d+1}}\,{\rm d}u
\leq  c_8R^{-(d+1)/2}. 
\end{equation}

Suppose next that $R<\sqrt{-2\lambda}t$.
If $v\geq (\sqrt{-2\lambda}t-R)/\sqrt{2t}$, then 
$$\frac{1}{(v+\sqrt{v^2+2\sqrt{-2\lambda}R})^d}\frac{1}{\sqrt{v^2+2\sqrt{-2\lambda}R}}
\leq c_1 R^{-(d+1)/2}$$
and thus
\begin{equation*}
\begin{split}
&\int_{(\sqrt{-2\lambda} t-R)/\sqrt{2t}}^{\infty}e^{-v^2}
\frac{1}{(v+\sqrt{v^2+2\sqrt{-2\lambda}R})^d}\frac{1}{\sqrt{v^2+2\sqrt{-2\lambda}R}}\,{\rm d}v\\
&\leq c_2R^{-(d+1)/2}\int_{(\sqrt{-2\lambda} t-R)/\sqrt{2t}}^{\infty} e^{-v^2}\,{\rm d}v.
\end{split}
\end{equation*}
Then by \eqref{eq:minus-part-2},
\begin{equation}\label{eq:positive-upper}
\int_{\sqrt{t}}^{\infty}e^{-(\sqrt{-\lambda}u-R/(\sqrt{2}u))^2}\frac{1}{u^{d+1}}\,{\rm d}u
\leq c_3 R^{-(d+1)/2}\int_{(\sqrt{-2\lambda} t-R)/\sqrt{2t}}^{\infty} e^{-v^2}\,{\rm d}v.
\end{equation}

\subsection*{Acknowledgments}
The authors would like to thank 
Professor Ekaterina VI. Bulinskaya for sending them 
the reference \cite{B20}.


\begin{thebibliography}{99}
\bibitem{ABM91}
S. Albeverio, P. Blanchard and Z. Ma,
Feynman-Kac semigroups in terms of signed smooth measures,
in ``Random Partial Differential Equations'' (U. Hornung et al.\, Eds.),
Birkh\"auser, Basel, 1991, pp.\ 1--31.







\bibitem{AN04}
K.~B. Athreya and P.~E. Ney,  
{\it Branching Processes},
Dover Publications, Inc., Mineola, NY, 2004. 




\bibitem{Be04} 
A. Ben Amor,  
Invariance of essential spectra for generalized Schr\"odinger operators, 
{\it Math.\ Phys.\ Electron.\ J.} {\bf 10} (2004), Paper 7, 18 pp.



\bibitem{BH14}
S. Bocharov and S.~C. Harris, 
Branching Brownian motion with catalytic branching at the origin,
{\it Acta.\ Appl.\ Math.} {\bf 134} (2014), 201--228.



\bibitem{BH16} 
S. Bocharov and S. C. Harris, 
Limiting distribution of the rightmost particle in catalytic branching Brownian motion, 
{\it Electron.\ Commun.\ Probab.} {\bf 21} (2016), no.\ 70, 12 pp. 


\bibitem{BW19}
S. Bocharov and L. Wang, 
Branching Brownian motion with spatially-homogeneous and point-catalytic branching,  
{\it J. Appl.\ Probab.} {\bf 56} (2019), 891--917.
 







\bibitem{B78}
M.~D. Bramson, 
Maximal displacement of branching Brownian motion,
{\it Comm.\ Pure Appl.\ Math.} {\bf 31} (1978), 531--581.



\bibitem{BEKS94}
J.~F. Brasche, P. Exner, Yu.~A. Kuperin and  P. \u{S}eba,
Schr\"odinger operators with singular interactions, 
{\it J. Math.\ Anal.\ Appl.} {\bf 184} (1994), 112--139.



\bibitem{B20}
E. Vl. Bulinskaya, 
Fluctuations of the propagation front of a catalytic branching walk, 
{\it Theory Probab.\ Appl.} {\bf 64} (2020), 513--534. 


\bibitem{CH14}
P. Carmona and  Y. Hu, 
The spread of a catalytic branching random walk,
{\it Ann.\ Inst.\ H. Poincar\'e Probab.\ Statist.} {\bf 50} (2014), 327--351.



\bibitem{CR88} 
B. Chauvin and A. Rouault,
KPP equation and supercritical branching Brownian motion in the subcritical speed area. 
Application to spatial trees, 
{\it Probab.\ Theory Related Fields.} {\bf 80} (1988), 299--314.



\bibitem{CR90} 
B. Chauvin and A. Rouault,
Supercritical branching Brownian motion and K-P-P equation in the critical speed-area,
{\it Math.\ Nachr.} {\bf 149} (1990), 41--59.


\bibitem{CRY17}
Z.-Q. Chen, Y.-X. Ren and T. Yang,
Law of large numbers for branching symmetric Hunt processes with measure-valued branching rates,
{\it J. Theoret.\ Probab.} {\bf 30} (2017), 898--931. 



\bibitem{CS07}
Z.-Q. Chen and Y. Shiozawa, 
Limit theorems for branching Markov processes,
{\it J. Funct.\ Anal.} {\bf 250} (2007), 374--399.

\bibitem{CZ95}
K.~L. Chung and Z. X. Zhao, 
{\it From Brownian Motion to Schr\"odinger's Equation}, Springer-Verlag, Berlin, 1995.


\bibitem{EHK10} 
J. Engl\"ander, S.~C Harris and A.~E. Kyprianou, 
Strong law of large numbers for branching diffusions, 
{\it Ann.\ Inst.\ H. Poincar\'e Probab.\ Statist.} {\bf 46} (2010), 279--298. 

\bibitem{E84}
K.~B. Erickson, 
Rate of expansion of an inhomogeneous branching process of Brownian particles,
{\it Z. Wahrsch.\ Verw.\ Gebiete} {\bf 66} (1984), 129--140. 


\bibitem{FOT11} 
M. Fukushima, Y. Oshima and M. Takeda, 
{\it Dirichlet Forms and Symmetric Markov Processes}, 2nd rev.\ and ext.\ ed., 
Walter de Gruyter, 2011.



\bibitem{HR17}
S. C. Harris and M. I. Roberts, 
The many-to-few lemma and multiple spines,
{\it Ann.\ Inst.\ H. Poincar\'e Probab.\ Statist.} {\bf 53} (2017), 226--242.









\bibitem{INW68-1} 
N. Ikeda, M. Nagasawa and S. Watanabe, 
Branching Markov processes \ I, 
{\it J. Math.\ Kyoto Univ.} {\bf 8} (1968), 233--278.  


\bibitem{INW68-2} 
N. Ikeda, M. Nagasawa and S. Watanabe, 
Branching Markov processes \ II, 
{\it J. Math.\ Kyoto Univ.} {\bf 8} (1968), 365--410.  

\bibitem{INW69} 
N. Ikeda, M. Nagasawa and S. Watanabe, 
Branching Markov processes \ III, 
{\it J. Math.\ Kyoto Univ.} {\bf 9} (1969), 95--160.

\bibitem{KS84}
I. Karatzas and S. E. Shreve, 
Trivariate density of Brownian motion, its local and occupation times, with application to stochastic control, 
{\it Ann.\ Probab.} {\bf 12} (1984), 819--828. 

\bibitem{K05} 
A. Kyprianou, 
Asymptotic radial speed of the support of supercritical branching Brownian motion 
and super-Brownian motion in ${\mathbb R}^d$, 
{\it Markov Process.\ Related Fields}, {\bf 11} (2005), 145--156. 


\bibitem{LS88}
S. Lalley and T. Sellke, 
Traveling waves in inhomogeneous branching Brownian motions. I, 
{\it Ann.\ Probab.} {\bf 16} (1988), 1051--1062. 

\bibitem{LS89}
S. Lalley and T. Sellke, 
Travelling waves in inhomogeneous branching Brownian motions. II, 
{\it Ann.\ Probab.} {\bf 17} (1989), 116--127.


\bibitem{M15}    
B. Mallein, 
Maximal displacement in the $d$-dimensional branching Brownian motion, 
{\it Electron.\ Commun.\ Probab.} {\bf 20} (2015), no.\ 76, 12 pp.


\bibitem{Mc75}
H.~P. McKean, 
Application of Brownian motion to the equation of Kolmogorov-Petrovskii-Piskunov, 
{\it Comm.\ Pure.\ Appl.\ Math.} {\bf 28} (1975), 323--331. 

\bibitem{Mc76}
H.~P. McKean, 
A correction to: Application of Brownian motion to the equation of Kolmogorov-Petrovskii-Piskunov, 
{\it Comm.\ Pure.\ Appl.\ Math.} {\bf 29} (1976), 553--554. 



\bibitem{RY99} 
D. Revuz and M. Yor, 
{\it Continuous Martingales and Brownian Motion}, 
Third edition, Springer-Verlag, Berlin, 1999. 


\bibitem{R13}
M.~I. Roberts, 
A simple path to asymptotics for the frontier of a branching Brownian motion, 
{\it Ann.\ Probab.} {\bf 41} (2013), 3518--3541.


\bibitem{R00}
A. Rouault, 
Large deviations and branching processes, 
{\it Pliska Stud.\ Math.\ Bulgar.} {\bf 13} (2000), 15--38. 


\bibitem{SBM20}
R. Stasi\'nski, J. Berestycki and B. Mallein,
Derivative martingale of the branching Brownian motion in dimension $d\geq 1$, 
preprint (2020), available at arXiv:2004.00162.


\bibitem{S08}
Y. Shiozawa,
Exponential growth of the numbers of particles 
for branching symmetric $\alpha$-stable processes,
{\it J. Math.\ Soc.\ Japan} {\bf 60} (2008), 75--116.

\bibitem{S18}
Y. Shiozawa, 
Spread rate of branching Brownian motions, 
{\it Acta Appl.\ Math.} {\bf 155} (2018), 113--150.


\bibitem{S19}
Y. Shiozawa,
Maximal displacement and population growth for branching Brownian motions, 
{\it Illinois J. Math.} {\bf 63} (2019), 353--402.

\bibitem{T03}
M. Takeda,
Large deviation principle for additive functionals of Brownian motion corresponding to Kato measures,
{\it Potential Anal.} {\bf 19} (2003), 51--67. 


\bibitem{T08}
M. Takeda, 
Large deviations for additive functionals of symmetric stable processes, 
{\it J. Theoret.\ Probab.} {\bf 21} (2008), 336--355.



\bibitem{WZ20}
L. Wang and G. Zong, 
Supercritical branching Brownian motion with catalytic branching at the origin, 
{\it Sci.\ China Math.} {\bf 63} (2020), 595--616.




\bibitem{W67} 
S. Watanabe, 
Limit theorems for a class of branching processes, 
in ``Markov Processes and Potential Theory" (J. Chover Eds.), Wiley, New York, 
1967, 205--232.

\end{thebibliography}
\end{document}